\documentclass[10pt, a4paper, reqno]{amsart}

\usepackage{bbm,amsmath,amssymb,amsfonts,graphicx,color,subfig,  
  enumerate,adjustbox,booktabs, multicol, bm, placeins}

\usepackage[%
  pdftex,
  plainpages=false,
  pdfpagelabels,
  colorlinks,
  citecolor=black,
  linkcolor=black,
  urlcolor=black,
  filecolor=black,
]{hyperref} 
  
\usepackage[all]{hypcap} 

\newcommand{\R}{\mathbb{R}}

\newcommand{\On}{\mathrm{O}(n)}

\newcommand{\Orm}{\mathrm{O}}
\newcommand{\Bcal}{\mathcal{B}}
\newcommand{\Ccal}{\mathcal{C}}
\newcommand{\Fcal}{\mathcal{F}}
\newcommand{\Rcal}{\mathcal{R}}
\newcommand{\Scal}{\mathcal{S}}

\newcommand{\Mcal}{\mathcal{M}}

\newcommand{\spann}{\mathrm{span}}
\newcommand{\sym}{\mathrm{sym}}
\newcommand{\Stab}[2]{\mathrm{Stab}_{#2}(#1)}
\DeclareMathOperator{\tr}{tr}
\newcommand*\diff{\mathop{}\!\mathrm{d}}

\newtheorem{defin}{Definition}[section]
\newtheorem{proposition}[defin]{Proposition}
\newtheorem{theorem}[defin]{Theorem}
\newtheorem{corollary}[defin]{Corollary}

\newtheorem{conjecture}[defin]{Conjecture}
\theoremstyle{remark}
\newtheorem{remark}[defin]{Remark}

\date{16 November 2019}

\title{$k$-point semidefinite programming bounds for equiangular lines}

\author{David de Laat}
\address{D. de Laat, Massachusetts Institute of Technology, Department of Mathematics, 77 
Massachusetts Avenue, Simons building (Building 2), Room 2-241}
\email{mail@daviddelaat.nl}
  
\author{Fabr\'\i cio Caluza Machado}
\address{F.C. Machado, Instituto de Matem\'atica e Estat\'\i stica, Rua do 
Mat\~ao 1010, 05508-090 S\~ao Paulo/SP, Brazil.}
\email{fabcm1@gmail.com}

\author{Fernando M\'ario de Oliveira Filho}
\address{F.M. de Oliveira Filho, Delft Institute of Applied
  Mathematics, Delft University of Technology, Van Mourik Broekmanweg
  6, 2628 XE Delft, The Netherlands.}
\email{fmario@gmail.com}

\author{Frank Vallentin}
\address{F.~Vallentin, Mathematisches Institut, Universit\"at zu
  K\"oln, Weyertal~86--90, 50931 K\"oln, Germany.}
\email{frank.vallentin@uni-koeln.de}


\thanks{D.d.L. was partially supported by NWO Rubicon grant 680-50-1528. 
  F.C.M was supported by grant \#2017/25237-4, from
  the S\~ao Paulo Research Foundation (FAPESP) and was financed in
  part by the Coordena\c{c}\~ao de Aperfei\c{c}oamento de Pessoal de
  N\'\i vel Superior --- Brasil (CAPES) --- Finance Code~001. F.V. was
  partially supported by the SFB/TRR~191 ``Symplectic Structures in
  Geometry, Algebra and Dynamics'', funded by the DFG\@. This project
  has received funding from the European Union’s Horizon 2020 research
  and innovation programme under the Marie Sk\l{}odowska-Curie
  agreement number~764759.}
  
\begin{document}

\begin{abstract}
We give a hierarchy of $k$-point bounds extending the Delsarte-Goethals-Seidel linear programming $2$-point bound and the Bachoc-Vallentin semidefinite programming $3$-point bound for spherical codes. An optimized implementation of this hierarchy allows us to compute~$4$, $5$, and $6$-point bounds for the maximum number of equiangular lines in Euclidean space with a fixed common angle.
\end{abstract}

\maketitle
\markboth{$k$-point semidefinite programming bounds for equiangular lines}{D. de Laat, F.C. Machado, F.M. de Oliveira Filho, and F. Vallentin}

\section{Introduction}

Given $D \subseteq [-1,1)$, a subset
$C$ of the unit sphere $S^{n-1} = \{\, x \in \R^n : \|x\|= 1\,\}$ is a \emph{spherical $D$-code} if
$x \cdot y \in D$ for all distinct $x$, $y \in C$, where~$x\cdot y$ is the Euclidean inner product between~$x$ and~$y$. The maximum cardinality of a spherical $D$-code in~$S^{n-1}$ is denoted by~$A(n,D)$.

A fundamental tool for computing upper bounds for~$A(n,D)$ is the \emph{linear programming bound} of Delsarte, Goethals, and Seidel~\cite{delsarte77}. If~$D = [-1, \cos(\pi/3)]$, then~$A(n, D)$ is the \emph{kissing number}, the maximum number of pairwise nonoverlapping unit spheres that can touch a central unit sphere; the linear programming bound was one of the first nontrivial upper bounds for the kissing number.

The linear programming bound is the optimal value of a convex optimization problem. It is a \emph{$2$-point bound}, because it takes into account interactions between pairs of points on the sphere: pairs~$\{x, y\}$ with~$x \cdot y \notin D$ correspond to constraints in the optimization problem. Bachoc and Vallentin~\cite{bachoc08} extended the linear programming bound to a \emph{$3$-point bound} by taking into account interactions between triples of points. The resulting \emph{semidefinite programming bound} gives the best known upper bounds for the kissing number in dimensions~$n = 5$, \dots,~$23$, except for dimension~$8$ where the linear programming bound is sharp.

In this paper, we give a hierarchy of \emph{$k$-point bounds} that extend both the linear and semidefinite programming bounds. The parameter~$A(n, D)$ is the independence number of a topological packing graph, namely the infinite graph with vertex set~$S^{n-1}$ in which two vertices~$x$ and~$y$ are adjacent if~$x\cdot y \notin D$. The linear programming bound corresponds to an extension of the Lov\'asz theta number to this infinite graph~\cite{bachoc09}. In Section~\ref{sec:hierarchy}, we derive our hierarchy from a generalization~\cite{laat15} of Lasserre's hierarchy to the independent set problem in these topological packing graphs. The first level of our hierarchy is the Lov\'asz theta number, and is therefore equivalent to the linear programming bound; the second level is the semidefinite programming bound, as shown in Section~\ref{sec:connection}.

Though our hierarchy is not as strong, in theory, as Lasserre's hierarchy, it is computationally less expensive. This allows us to use it to compute $4$, $5$, and~$6$-point bounds for the maximum number of equiangular lines with a certain angle, a problem that corresponds to the case~$|D|=2$. Aside from a previous result of de Laat~\cite{laat19}, which uses Lasserre's hierarchy directly, this is the first successful use of $k$-point bounds for~$k > 3$ for geometrical problems; it yields improved bounds for the number of equiangular lines with given angles in several dimensions.

To perform computations, we transform the resulting problems into semidefinite programming problems. To this end, for a given~$k \geq 2$ we use a characterization of kernels~$K\colon S^{n-1} \times S^{n-1} \to \R$ on the sphere that are invariant under the action of the subgroup of the orthogonal group that stabilizes~$k-2$ given points. For~$k=2$, this characterization was given by Schoenberg~\cite{schoenberg42} and for~$k = 3$, by Bachoc and Vallentin~\cite{bachoc08}; Musin~\cite{musin08} extended these two results for~$k > 3$; a similar result is given by Kuryatnikova and Vera \cite{kuryatnikova19}.

Our hierarchy is particularly suited for problems like the equiangular lines problem, where~$D$ is finite. Indeed, let~$I_k$ be the set of all spherical $D$-codes of cardinality at most~$k$, and let the orthogonal group act on~$S \in I_k$ by rotating each point in~$S$. The final semidefinite programming problem has one block corresponding to each orbit of~$I_k$ under this action. So, if~$k \leq 3$ or~$D$ is finite, then the number of orbits is finite and we get a finite semidefinite programming problem.

Still, a naive implementation of our approach would be too slow even to generate the problems for~$k = 5$. The implementation available with the arXiv version of this paper was carefuly written to deal with the orbits of~$I_k$ in an efficient way; this allows us to generate problems even for~$k = 6$. This implementation could be of interest to others working on similar problems.

\subsection{Equiangular lines}

A set of \textit{equiangular lines} is a set of lines through the origin
such that every pair of lines defines the same
angle. If this angle is~$\alpha$, then such a set of equiangular lines corresponds to a spherical $D$-code where~$D = \{a, -a\}$ and~$a = \cos\alpha$. So we are interested in finding~$A(n, \{a,-a\})$ for a given~$a \in [-1,1)$ and also in finding the maximum number of equiangular lines with any given angle, namely
\[
    M(n) = \max\{\, A(n, \{a,-a\}) : a \in [-1,1)\,\}.
\]

The study of~$M(n)$
started with Haantjes \cite{haantjes48} in 1948. He showed that~$M(2) = 3$ and that the optimal configuration is a set of lines in the plane having a common angle of~$60$ degrees. He also showed that~$M(3) = 6$; the optimal configuration is given by the lines going through opposite vertices of a
regular icosahedron, which have a common angle of~$63.43\ldots$
degrees.  These two constructions provide lower bounds; in both cases,
Gerzon's bound, which states that~$M(n) \leq n(n+1)/2$ (see
Theorem~\ref{thm:gerzon} below which is proven for example in
Matou\v{s}ek's book~\cite[Miniature 9]{matousek10}), provides matching
upper bounds.

In the setting of equiangular lines, the LP bound coincides with the
relative bound (see Theorem~\ref{thm:lint}). The $3$-point SDP bound
was first specialized to this setting by Barg and Yu~\cite{barg13}. 
Gijswijt, Mittelmann,
and Schrijver~\cite{gijswijt12} computed $4$-point SDP bounds for
binary codes and Litjens, Polak, and Schrijver~\cite{litjens17} extended
these $4$-point bounds for $q$-ary codes. For spherical codes, no $k$-point bounds have been computed or formulated for $k \ge 4$.

Next to being fundamental objects in discrete geometry, equiangular
lines have applications, for example in the field of compressed
sensing: Only measurement matrices whose columns are unit vectors
determining a set of equiangular lines can minimize the coherence
parameter~\cite[Chapter 5]{foucart13}.

In general, it is a difficult problem to determine~$M(n)$ for a given dimension~$n$. Currently, the
first open case is dimension $n = 14$ where it is known that $M(14)$
is either $28$ or $29$; see Table~\ref{tab:Mn-small}. Sequence A002853
in The On-Line Encyclopedia of Integer Sequences~\cite{sloane}
is~$M(n)$.

\section{Derivation of the hierarchy}\label{sec:hierarchy}

Let $G = (V,E)$ be a graph. A subset of~$V$ is \textit{independent} if
it does not contain a pair of adjacent vertices. The
\textit{independence number} of~$G$, denoted by~$\alpha(G)$, is the
maximum cardinality of an independent set. For an integer $k \geq 0$,
let $I_{k}$ be the set of independent sets in $G$ of size at most~$k$
and $I_{=k}$ be the set of independent sets in $G$ of size
exactly~$k$.

Assume for now that~$G$ is finite. We can obtain upper bounds for the
independence number via the Lasserre hierarchy~\cite{lasserre02} for
the independent set problem, whose $t$-th step can be formulated as
\begin{align*}
\max \bigg\{\,\sum_{S \in I_{=1}} \nu_S : \nu \in \R^{I_{2t}}_{\geq 0}, \
\nu_{\emptyset} = 1,\text{ and } M(\nu) \succeq 0\,\bigg\},
\end{align*}
where $M(\nu)$ is the matrix indexed by $I_t \times I_t$ such that
\[
M(\nu)_{J,J'} = \begin{cases} \nu_{J \cup J'} & \text{if } J\cup J' \text{ is 
independent};\\
0 & \text{otherwise}
\end{cases}
\]
and $M(\nu) \succeq 0$ means that~$M(\nu)$ is positive semidefinite.
To produce an optimization program where the variables are easier to
parameterize, we construct in two stages a weaker hierarchy with
matrices indexed only by the vertex set of the graph. First, we modify
the problem to remove $\emptyset$ from the domain of~$\nu$; this gives
the possibly weaker problem
\begin{align*}
\max \bigg\{\,1+2\sum_{S \in I_{=2}} \nu_S : \nu \in \R^{I_{2t} \setminus 
  \{\emptyset\}}_{\geq 0}, \ \sum_{S \in I_{=1}} \nu_S = 1,\text{ and } M(\nu) \succeq 
  0\,\bigg\},
\end{align*}
where $M(\nu)$ is now considered as a matrix indexed by
$(I_t \setminus \{\emptyset\}) \times (I_t \setminus \{\emptyset\})$.

Second, we construct a weaker hierarchy by only requiring certain
principal submatrices of $M(\nu)$ to be positive semidefinite, an
approach similar to the one employed by Gvozdenovi\'c, Laurent, and
Vallentin~\cite{gvozdenovic09}.  For this we fix~$k \geq 2$ and, for
each $Q \in I_{k-2}$, define the matrix
$M_Q(\nu)\colon V \times V \to \R$ by
\[
M_Q(\nu)(x,y) = \begin{cases}
\nu_{Q \cup \{x, y\}} & \text{if $Q \cup \{x, y\} \in I_{k}$};\\
0 & \text{otherwise}
\end{cases}
\]
and replace the condition `$M(\nu) \succeq 0$' by
`$M_Q(\nu) \succeq 0$ for all $Q \in I_{k-2}$'. With these conditions
we can restrict the support of $\nu$ to the set
$I_k \setminus \{\emptyset\}$, obtaining the relaxation
\begin{align}\label{prog:stepk-fin}
  \max \bigg\{\,1+2\sum_{S \in I_{=2}} \nu_S : \nu \in \R^{I_k \setminus 
  \{\emptyset\}}_{\geq 0}, \ \sum_{S \in I_{=1}} \nu_S = 1,\text{ and
  } M_Q(\nu) \succeq 0  \text{ for } Q \in I_{k-2}\,\bigg\}.
\end{align}

We extend this problem to infinite graphs in the same way that the
Lasserre hierarchy is extended by de Laat and Vallentin~\cite{laat15}.
This extension can be carried out for \textit{compact topological
  packing graphs}; these are graphs whose vertex sets are compact
Hausdorff spaces and in which every finite clique is contained in an
open clique. The independence number of a compact topological packing
graph is finite and~$I_k$, considered with the topology inherited from
$V$, is the disjoint union of the compact and open sets~$I_{=s}$ for
$s = 0$, \dots,~$k$ \cite[Section 2]{laat15}.

The extension relies on the theory of conic optimization over
infinite-dimensional spaces presented e.g.\ by
Barvinok~\cite{barvinok02}. The first step is to introduce the spaces
for the variables of our problem; we will use both the
space~$\Ccal(X)$ of continuous real-valued functions on a compact
space~$X$ and its topological dual (with respect to the supremum norm)
$\Mcal(X)$, the space of signed Radon measures.

In the infinite setting, the nonnegative variable~$\nu$
from~\eqref{prog:stepk-fin} becomes a measure in the dual of the cone
$\Ccal(I_k \setminus \{\emptyset\})_{\geq 0}$ of continuous and
nonnegative functions, namely
\[
  \Mcal(I_k \setminus \{\emptyset\})_{\geq 0} = \{\,\nu \in \Mcal(I_k 
\setminus \{\emptyset\}) : \nu(f) \geq 0 \text{ for all } f \in \Ccal(I_k 
\setminus \{\emptyset\})_{\geq 0}\,\};
\]
when~$V$ is finite, $\Mcal(I_k \setminus \{\emptyset\})_{\geq 0}$ can
be identified with $\R^{I_k \setminus \{\emptyset\}}_{\geq 0}$.

Let $\Ccal(V^2 \times I_{k-2})_{\sym}$ be the set of continuous
real-valued functions on $V^2 \times I_{k-2}$ that are symmetric in
the first two coordinates and let $\Mcal(V^2 \times I_{k-2})_{\sym}$
be the space of symmetric and signed Radon measures\footnote{A measure
  $\mu \in \Mcal(V^2 \times I_{k-2})$ is \emph{symmetric} if
  $\mu(E \times E' \times C) = \mu(E' \times E \times C)$ for all
  Borel sets $E$, $E' \subseteq V$ and $C \subseteq I_{k-2}$.}. A
kernel~$K \in \Ccal(V^2)$ is \textit{positive} if for every
finite~$U \subseteq V$ the matrix~$\bigl(K(x, y)\bigr)_{x,y \in U}$ is
positive semidefinite. A function~$T \in \Ccal(V^2 \times I_{k-2})$ is
\textit{positive} if for every~$Q \in I_{k-2}$ the
kernel~$(x, y) \mapsto T(x, y, Q)$ is positive. The set of all
positive functions in~$\Ccal(V^2 \times I_{k-2})$ is a convex cone
denoted by $\Ccal(V^2 \times I_{k-2})_{\succeq 0}$; its dual cone is
denoted by $\Mcal(V^2 \times I_{k-2})_{\succeq 0}$.

Define the operator
$B_k \colon \Ccal(V^2 \times I_{k-2})_{\sym} \to \Ccal(I_k \setminus
\{\emptyset\})$ by
\begin{equation}\label{eq:operatorBk}
B_k T(S) = \sum_{\substack{Q \subseteq S \\ |Q| \leq k-2}} \sum_{\substack{x, y 
\in S \\ Q \cup \{x, y\} = S}} T(x, y, Q).
\end{equation}
Note that, though the number of summands in \eqref{eq:operatorBk}
varies with the size of $S$, the function~$B_k T$ is still continuous
since, by the assumption that~$G$ is a topological packing graph,
$I_k \setminus \{\emptyset\}$ can be written as the disjoint union of
the compact and open subsets $I_{=s}$ for $s = 1$, \dots,~$k$ and
$B_kT$ is continuous in each of these parts. When~$V$ is finite, $B_k$
is the adjoint of the operator that maps
$\nu \in \R^{I_k \setminus \{\emptyset\}}_{\geq 0}$ to the tuple of
matrices $\big(M_Q(\nu)\big)_{Q \in I_{k-2}}$, since the inner product
between $(x,y,Q) \mapsto M_Q(\nu)(x,y)$ and $T$ is equal to the inner
product between $\nu$ and~$B_kT$:
\begin{align*}
  \sum_{Q \in I_{k-2}} \sum_{x,y \in V} M_Q(\nu)(x,y) T(x,y,Q) 
  &= \sum_{Q \in I_{k-2}} \; \sum_{\substack{x,y \in V \\ Q \cup
  \{x,y\} \in I_k}} \nu_{Q \cup \{x,y\}} T(x,y,Q)\\ 
  &= \sum_{S \in I_k\setminus \{\emptyset\}} \nu_{S} B_kT(S).
\end{align*}
So, when~$V$ is finite, we may rewrite the constraints
`$M_Q(\nu) \succeq 0$ for all~$Q \in I_{k-2}$' from
\eqref{prog:stepk-fin} as
`$B_k^*\nu \in \Mcal(V^2 \times I_{k-2})_{\succeq 0}$'.

This last observation leads us naturally to the generalized $k$-point
bound. Indeed, when~$G = (V, E)$ is a compact topological packing
graph, since the number of summands in \eqref{eq:operatorBk} is
bounded by a constant depending only on~$k$, the operator $B_k$ is
continuous. Thus it has an adjoint
$B_k^* \colon \Mcal(I_k \setminus \{\emptyset\}) \to \Mcal(V^2 \times
I_{k-2})_{\sym}$.  Using the adjoint, we define the
\textit{generalized $k$-point bound} for~$k \geq 2$:
\begin{multline}
 \label{prog:k-point-definition}
\Delta_k(G) = \sup \{\,1 + 2\nu(I_{=2}) : \nu \in \Mcal(I_k \setminus
\{\emptyset\})_{\geq 0},\\
\nu(I_{=1}) = 1,\text{ and } B_k^*\nu \in \Mcal(V^2 \times
I_{k-2})_{\succeq 0}\, \}.
\end{multline}
For a finite graph with the discrete topology this reduces to
\eqref{prog:stepk-fin}.

Using the duality theory of conic optimization as described e.g.\ by
Barvinok~\cite[Chapter IV]{barvinok02}, we can derive the following
dual problem for~\eqref{prog:k-point-definition}:
\begin{equation}\label{prog:stepk-top-dual}
  \Delta_k(G)^* = \inf \{\,1 + \lambda : \lambda \in \R,\ T \in
  \Ccal(V^2 \times I_{k-2})_{\succeq 0},\text{ and } B_kT \leq \lambda
  \chi_{I_{=1}} - 2 \chi_{I_{=2}}\,\}, 
\end{equation}
where $\chi_{I_{=1}}$ and $\chi_{I_{=2}}$ are the characteristic
functions of $I_{=1}$ and $I_{=2}$, which are continuous since $G$ is
a topological packing graph. From now on, we will denote both the
optimal value of~\eqref{prog:stepk-top-dual} and the optimization
problem itself by~$\Delta_k(G)^*$.

It is a direct consequence of weak duality that $\Delta_k(G)^*$ is an
upper bound for the independence number of $G$, but it is instructive
to see a direct proof.

\begin{proposition}\label{prop:upperbound}
  If $G = (V,E)$ is a compact topological packing graph, then
  $\alpha(G) \leq \Delta_k(G)^*$.
\end{proposition}
\begin{proof}
  Let $C \subseteq V$ be a nonempty independent set and let
  $(\lambda, T)$ be a feasible solution of $\Delta_k(G)^*$.  On the
  one hand, since $B_kT \leq \lambda \chi_{I_{=1}} - 2 \chi_{I_{=2}}$, we have
\begin{equation*}
 \sum_{\substack{S \subseteq C\\ |S| \leq k,\ S \neq \emptyset}} B_{k}T(S) \leq 
\binom{|C|}{1}\lambda + \binom{|C|}{2}(-2) = |C|(1 + \lambda - |C|).
\end{equation*}
On the other hand, since
$T \in \Ccal(V^2 \times I_{k-2})_{\succeq 0}$, we have
\begin{equation*}
\begin{aligned}
\sum_{\substack{S \subseteq C\\ |S| \leq k,\ S \neq \emptyset}} B_{k}T(S) 
&= \sum_{\substack{S \subseteq C\\ |S| \leq k,\ S \neq \emptyset}}\;
\sum_{\substack{Q \subseteq S \\ |Q| \leq k-2}}\; \sum_{\substack{x, y 
\in S \\ Q \cup \{x, y\} = S}} T(x,y,Q)\\ 
&= \sum_{\substack{Q \subseteq C\\ |Q| \leq k-2}}\; 
\sum_{x, y \in C} T(x,y,Q) \geq 0
\end{aligned}
\end{equation*}
since, by the definition of $\Ccal(V^2 \times I_{k-2})_{\succeq 0}$,
the matrices $\bigl(T(x,y,Q)\bigr)_{x,y\in C}$ are positive
semidefinite for all $Q \in I_{k-2}$. Putting it all together we get
$|C| \leq 1 + \lambda$.
\end{proof}

\section{Parameterizing the variables by positive semidefinite matrices}
\label{sec:symmetry}

Symmetry reduction plays a key role in the computation of
$\Delta_k(G)^*$ via semidefinite programming. We now see how to
exploit symmetry to parameterize the variable~$T$
of~\eqref{prog:stepk-top-dual} in terms of positive semidefinite
matrices, an essential step in the solution of these optimization
problems by computer.

\subsection{Symmetry reduction}

Let $\Gamma$ be a compact group that acts continuously on~$V$ and that
is a subgroup of the automorphism group\footnote{The
  \textit{automorphism group} $\mathrm{Aut}(G)$ of a graph
  $G = (V, E)$ is the group of permutations~$\sigma\colon V \to V$
  that respect the adjacency relation, that is, $\sigma(x)$ and
  $\sigma(y)$ are adjacent if and only if~$x$ and $y \in V$ are
  adjacent.} of the graph $G$. The group~$\Gamma$ acts coordinatewise
on~$V^2$, and this action extends to an action on~$\Ccal(V^2)$ by
\[
(\gamma K)(x, y) = K(\gamma^{-1} x, \gamma^{-1} y).
\]
The group $\Gamma$ acts continuously on $I_t$ by
\[
  \gamma \emptyset = \emptyset \quad \text{and} \quad \gamma
  \{x_1,\ldots,x_t\} = \{\gamma x_1,\ldots,\gamma x_t\},
\]
and hence it also acts on~$\Ccal(V^2 \times I_{k-2})_{\sym}$ by
\[
(\gamma T)(x, y, S) = T(\gamma^{-1} x, \gamma^{-1} y, \gamma^{-1} S).
\]

If $\Gamma$ acts on a set $X$, we denote by $X^\Gamma$ the set of
elements of~$X$ that are invariant under this action. In this way we
write $\Ccal(V^2)^{\Gamma}$,
$\Ccal(V^2 \times I_{k-2})_{\succeq 0}^{\Gamma}$, etc.

Given a feasible solution $(\lambda, T)$ of $\Delta_k(G)^*$, the pair
$(\lambda, \overline{T})$ with
\[
\overline{T}(x,y,S)=\int_\Gamma T(\gamma^{-1} x, \gamma^{-1} y,
\gamma^{-1} S) \, \diff\gamma, 
\]
where we integrate against the Haar measure on~$\Gamma$ normalized so
that the total measure is~1, is also feasible with the same objective
value. So we may assume that~$T$ is invariant under the action
of~$\Gamma$.

Let $\Rcal_{k-2}$ be a complete set of representatives of the orbits
of $I_{k-2}/\Gamma$. For $R \in \Rcal_{k-2}$, let
$\Stab{R}{\Gamma} = \{\,\gamma \in \Gamma : \gamma R = R\,\}$ be the
stabilizer of $R$ with respect to $\Gamma$ and, for $Q \in \Gamma R$,
let $\gamma_Q \in \Gamma$ be a group element such that
$\gamma_Q Q = R$. When $I_{k-2}/\Gamma$ is finite, we can decompose
the space $\Ccal(V^2 \times I_{k-2})^\Gamma$ as a direct sum of
simpler spaces.

\begin{proposition}\label{prop:isomorphism}
If $I_{k-2}/\Gamma$ is finite, then
\[
\Psi \colon \bigoplus_{R \in \Rcal_{k-2}}
\Ccal(V^2)^{\Stab{R}{\Gamma}} \to \Ccal(V^2 \times I_{k-2})^\Gamma 
\]
given by
\begin{equation*}
\Psi ((K_R)_{R \in \Rcal_{k-2}})(x, y, Q) = K_{\gamma_Q Q}(\gamma_Q
x,\gamma_Q y) 
\end{equation*}
is an isomorphism that moreover preserves positivity, that is,
if~$(K_R)_{R \in \Rcal_{k-2}}$ is such that~$K_R$ is a positive kernel
for each~$R$, then~$\Psi((K_R)_{R \in \Rcal_{k-2}})$ is positive.
\end{proposition}
\begin{proof}
We first show that $(V^2 \times I_{k-2}) / \Gamma$ is homeomorphic 
to the coproduct
\[
\coprod_{R \in \Rcal_{k-2}} V^2 / \Stab{R}{\Gamma},
\]
which is the union of $(V^2 / \Stab{R}{\Gamma}) \times \{R\}$ over all
$R \in \Rcal_{k-2}$, endowed with the disjoint union topology. More
precisely, we show that
\[
\psi \colon \coprod_{R \in \Rcal_{k-2}} V^2 /\Stab{R}{\Gamma} \to (V^2
\times I_{k-2}) / \Gamma
\]
given by $\psi(\Stab{R}{\Gamma}(x,y),R) = \Gamma(x,y,R)$ is such a
homeomorphism with inverse
\begin{equation}
  \label{eqpsiinverse}
\psi^{-1}(\Gamma(x,y,Q)) = (\Stab{\gamma_Q
  Q}{\Gamma}(\gamma_Q x, \gamma_Q y), \gamma_QQ). 
\end{equation}

Indeed, the map $\psi$ is well defined because
$\Gamma(x,y,R) = \Gamma(\gamma x, \gamma y, R)$ for all $\gamma$ in
$\Stab{R}{\Gamma}$. For each $R \in \mathcal R_{k-2}$, the map
$\psi_R \colon V^2 / \Stab{R}{\Gamma} \to (V^2 \times I_{k-2}) /
\Gamma$ given by
\[
\psi_R\big(\Stab{R}{\Gamma}(x,y)\big) = \Gamma(x,y,R)
\]
is continuous, as follows from the definition of quotient topology. By
the definition of disjoint union topology on the coproduct this
implies $\psi$ is continuous.

The map \eqref{eqpsiinverse} is well defined, for if we replace
$\gamma_Q$ by $\xi\gamma_Q$, where
$\xi \in \Stab{\gamma_Q Q}{\Gamma}$, then the right-hand side of
\eqref{eqpsiinverse} does not change. Direct verification shows
$\psi^{-1} \circ \psi$ and $\psi \circ \psi^{-1}$ are the identity
maps.

Since $\mathcal R_{k-2}$ is finite, the domain of $\psi$ is
compact. So $\psi$ is a continuous bijection between compact Hausdorff
spaces, and hence a homeomorphism.

Now the proposition follows easily. Under the isomorphisms
\[
\Ccal\biggl(\coprod_{R \in \mathcal R_{k-2}} V^2 /\Stab{R}{\Gamma}\biggr) 
\simeq \bigoplus_{R \in \mathcal R_{k-2}} \Ccal(V^2)^{\Stab{R}{\Gamma}}
\]
and
\[
\Ccal((V^2 \times I_{k-2}) / \Gamma) \simeq \Ccal(V^2 \times I_{k-2})^\Gamma,
\]
the operator $\Psi$ is equal to
\[
\Ccal\biggl(\coprod_{R \in \mathcal R_{k-2}} V^2 /\Stab{R}{\Gamma}\biggr) \to 
\Ccal((V^2 \times I_{k-2}) / \Gamma), \qquad f \mapsto f \circ \psi^{-1},
\]
which is a well-defined isomorphism since $\psi$ is a
homeomorphism. Finally, it follows directly from the definitions of
positive kernels and~$\Ccal(V^2 \times I_{k-2})^\Gamma_{\succeq 0}$
that~$\Psi$ preserves positivity.
\end{proof}

The above proposition shows that to characterize
$\Ccal(V^2 \times I_{k-2})^\Gamma$ we need to characterize the sets
$\Ccal(V^2)^{\Stab{R}{\Gamma}}$ for~$R \in \Rcal_{k-2}$. In the next
section we give this characterization for the case of spherical
symmetry.

\subsection{Parameterizing invariant kernels on the sphere}
\label{sec:parameterizing} 

From now on we assume $G = (V,E)$ is the graph where $V = S^{n-1}$ and
where two distinct vertices $x$, $y \in S^{n-1}$ are adjacent if
$x \cdot y \notin D$ for some closed $D \subseteq [-1, 1)$. Taking
$\Gamma = \On$, we are in the situation described above. Let us see
how to parameterize the cones
\[
\Ccal(S^{n-1} \times S^{n-1})_{\succeq 0}^{\Stab{R}{\On}} \quad \text{for }
R \in \Rcal_{k-2}
\]
by positive semidefinite matrices. For notational simplicity, we make
the assumption that every $R \in \Rcal_{k-2}$ consists of linearly
independent vectors, which is true for all $G$ and $k$ considered in
this paper.

Let $\{e_1,\ldots,e_n\}$ be the standard basis of $\R^n$ and fix
$R \in \Rcal_{k-2}$. By rotating a set~$R \in \Rcal_{k-2}$ if
necessary, we may assume that $R$ is contained in
$\spann(\{e_1,\ldots, e_m\})$, where $m = \dim(\spann(R))$. The
stabilizer subgroup of $\Orm(n)$ with respect to $R$ is isomorphic to
the direct product of two groups, namely
\[
\Stab{R}{\Orm(n)} \simeq \Scal_R \times \Stab{\spann(R)}{\Orm(n)},
\]
where $\Scal_R$ is isomorphic to a finite subgroup of $\Orm(m)$ that
acts on the first $m$ coordinates and acts on $R$ as a permutation of
its elements and $\Stab{\spann(R)}{\On}$ is a group isomorphic to
$\Orm(n-m)$ that acts on the last $n-m$ coordinates.

If $k=2$, then $R = \emptyset$ and $\Stab{\spann(R)}{\On} = \On$. By a
classical result of Schoenberg~\cite{schoenberg42}, each positive
$\On$-invariant kernel $K \colon S^{n-1} \times S^{n-1} \to \R$ is of
the form
\[
K(x,y) = \sum_{l=0}^\infty a_l P_l^n(x \cdot y)
\]
for some nonnegative numbers $a_0$, $a_1$, \dots\ with absolute and
uniform convergence, where $P_l^n$ is the Gegenbauer polynomial of
degree $l$ in dimension $n$ normalized so that $P_l^n(1) = 1$
(equivalently, $P_l^n$ is the Jacobi polynomial with both parameters
equal to $(n-3)/2$).

Kernels invariant under the stabilizer of one point were considered by
Bachoc and Vallentin~\cite{bachoc08} and kernels invariant under the
stabilizer of more points were considered by Musin~\cite{musin14}. The
analogue of Schoenberg's theorem for kernels invariant under the
stabilizer of one or more points is stated in terms of multivariate
Gegenbauer polynomials.

For $0 \leq m \leq n - 2$, $t \in \R$, and $u$, $v \in \R^m$, the
\emph{multivariate Gegenbauer polynomial}~$P_l^{n,m}$ is the
$(2m+1)$-variable polynomial
\[ 
  P_l^{n, m}(t,u,v) =
  \big((1-\|u\|^2)(1-\|v\|^2)\big)^{l/2}P_l^{n-m}\bigg(\frac{t-u \cdot
    v}{\sqrt{(1-\|u\|^2)(1-\|v\|^2)}}\bigg),
\]
where $\|v\| = \sqrt{v \cdot v}$. If we use the convention
$\R^0 = \{0\}$, then $P_l^n(t) = P_l^{n,0}(t,0,0)$.  Since the
Gegenbauer polynomials are odd or even according to the parity of $l$,
the function $P_l^{n, m}(t,u,v)$ is indeed a polynomial in the
variables $u$, $v$, and $t$. Musin~\cite{musin14} denotes $P_l^{n,m}$
by $G_l^{(n,m)}$ and Bachoc and Vallentin~\cite{bachoc08} denote
$P_l^{n,1}$ by $Q_l^{n-1}$.

Fix $d \geq 0$, let $\Bcal_l$ be a basis of the space of $m$-variable
polynomials of degree at most~$l$ (e.g.\ the monomial basis), and
write $z_l(u)$ for the column vector containing the polynomials in
$\Bcal_l$ evaluated at $u \in \R^m$. Let $Y_l^{n, m}$ be the matrix of
polynomials
\[
  Y_l^{n,m}(t,u,v) = P_l^{n,m}(t,u,v) z_{d-l}(u)z_{d-l}(v)^{\sf T}.
\]
The choice of $d$ makes $Y_l^{n,m}$ a
$\binom{d-l+m}{m} \times \binom{d-l+m}{m}$ matrix with
$(2m+1)$-variable polynomials of degree at most~$2d$ as its entries.

Given a matrix $X$ with linearly independent columns, set
$L(X) = B^{-1} X^{\sf T}$, where $B$ is the matrix such that
$BB^{\sf T}$ is the Cholesky factorization of $X^{\sf T}X$, which is
unique since $X^{\sf T}X$ is positive definite. For each
$R \in \mathcal R_{k-2}$, fix a matrix $A_R$ whose columns are the
vectors of $R$ in some order. The rows of $L(A_R)$ span the same space
as the columns of $A_R$ because $B$ is invertible, and by construction
the rows of $L(A_R)$ are orthonormal:
\[
  L(A_R) L(A_R)^{\sf T} = B^{-1} A_R^{\sf T} A_R B^{-\sf T} = B^{-1}
  BB^{\sf T} B^{-\sf T} = I.
\]
Therefore, for $x \in \R^n$, $L(A_R)x$ is a vector with the
coordinates of the projection of~$x$ onto $\spann(R)$ with respect to
an orthonormal basis of the linear span.

The following theorem is a restatement of a result of
Musin~\cite[Corollary 3.2]{musin14} in terms of invariant kernels and
with adapted notation. We will use only the sufficiency part of the
statement, proved in Appendix~\ref{ap:gegenbauer} for completeness.

For square matrices $A$, $B$ of the same dimensions, write
$\langle A, B \rangle = \tr(A^{\sf T}B)$ for their trace product.

\begin{theorem}\label{teo:main-charac}
  Let $R \subseteq S^{n-1}$ with $m = \dim(\spann(R)) = |R| \leq n-2$
  and let $A_R$ be a matrix whose columns are the vectors of $R$ in
  some order. Fix $d \geq 0$ and, for each $0 \leq l \leq d$, let
  $F_l$ be a positive semidefinite matrix of size
  $\binom{d-l+m}{m} \times \binom{d-l+m}{m}$. Then
  $K \colon S^{n-1} \times S^{n-1} \to \R$ given by
\begin{equation}\label{eq:K-thm}
 K(x,y) = \sum_{l = 0}^d \big\langle F_l, Y_l^{n,m}(x \cdot y, L(A_R) x, 
L(A_R) y) \big\rangle
\end{equation}
is a positive, continuous, and $\Stab{\spann(R)}{\On}$-invariant
kernel.  Conversely, every $\Stab{\spann(R)}{\On}$-invariant kernel
$K \in \Ccal(S^{n-1} \times S^{n-1})_{\succeq 0}$ can be uniformly
approximated by kernels of the above form.
\end{theorem}

Theorem~\ref{teo:main-charac} gives us a parameterization of
$\Stab{\spann(R)}{\On}$-invariant kernels. To get a parameterization
of $\Stab{R}{\On}$-invariant kernels we still have to deal with the
symmetries in $\Scal_R$. By construction, for an orthogonal matrix
$\sigma \in \Scal_R$ there is a permutation matrix $P_\sigma$ such
that $\sigma A_R = A_R P_\sigma$. Since $\sigma \in \On$ and
$A_R^{\sf T}A_R = A_R^{\sf T}\sigma^{\sf T}\sigma A_R = P_\sigma^{\sf
  T}A_R^{\sf T}A_RP_\sigma$, the elements of $\Scal_R$ correspond
precisely to the symmetries of the Gram matrix $A_R^{\sf T}A_R$ under
simultaneous permutations of rows and columns. Indeed, if the Gram
matrix $A_R^{\sf T}A_R$ is invariant under a certain simultaneous
permutation of rows and columns, then since $R$ is linearly
independent, this permutation defines a linear transformation of
$\spann(R)$ that preserves all inner products between vectors of $R$,
whence it is orthogonal and therefore corresponds to an element
of~$\Scal_R$. This observation leads to the following corollary.

\begin{corollary}\label{cor:K-inv}
  Let $R \subseteq S^{n-1}$ with $m = \dim(\spann(R)) = |R| \leq n-2$
  and let $A_R$ be a matrix whose columns are the vectors of $R$ in
  some order. Fix $d \geq 0$ and, for each $0 \leq l \leq d$, let
  $F_l$ be a positive semidefinite matrix of size
  $\binom{d-l+m}{m} \times \binom{d-l+m}{m}$. Then
  $K \colon S^{n-1} \times S^{n-1} \to \R$ given by
\begin{equation}
\label{eq:K-StabR}
K(x,y) = \sum_{l = 0}^d \big\langle F_l,\Fcal_l(R)(x,y) \big\rangle,
\end{equation}
where
\[
\Fcal_l(R)(x,y) = \frac{1}{|\Scal_R|} \sum_{\sigma \in \Scal_R}
Y_l^{n,m}(x \cdot y, L(A_R P_\sigma) x, L(A_RP_\sigma)y),
\]
is a positive, continuous, and $\Stab{R}{\On}$-invariant kernel.
\end{corollary}

\begin{proof}
If~$K$ is given by~\eqref{eq:K-StabR}, then by writing
\[
  K(x,y) = \frac{1}{|\Scal_R|} \sum_{\sigma \in \Scal_R} \sum_{l =
    0}^d \big\langle F_l, Y_l^{n,m}(x \cdot y, L(A_R P_\sigma) x,
  L(A_R P_\sigma) y) \big\rangle
\]
we see using Theorem~\ref{teo:main-charac} that~$K$ is a sum of
$|\Scal_R|$ positive, continuous, and
$\mathrm{Stab}_{\On}\break(\spann(R))$-invariant kernels, and hence it is itself
positive, continuous, and $\mathrm{Stab}_{\On}\break(\spann(R))$-invariant.

Since, for every $\sigma \in \Scal_R$,
\[
  L(A_RP_{\sigma})x = B^{-1}P_{\sigma}^{\sf T}A_R^{\sf T}x =
  B^{-1}A_R^{\sf T} \sigma^{\sf T} x = L(A_R)\sigma^{\sf T} x
\]
(recall $BB^{\sf T}$ is the Cholesky decomposition of
$A_R^{\sf T} A_R = (A_R P_\sigma)^{\sf T}(A_R P_\sigma)$), and since
$x \cdot y = (\sigma^{\sf T} x) \cdot (\sigma^{\sf T} y)$, we have
that
\begin{equation}
  \label{eq:K-SR-symmetrization}
  K(x,y) = \frac{1}{|\Scal_R|} \sum_{\sigma \in \Scal_R} K'(\sigma^{\sf T} x,
  \sigma^{\sf T} y),
\end{equation}
where
\[
  K'(x,y) = \sum_{l = 0}^d \big\langle F_l, Y_l^{n,m}(x \cdot y,
  L(A_R) x, L(A_R) y) \big\rangle.
\]
Now it follows directly from~\eqref{eq:K-SR-symmetrization} that~$K$
is $\Stab{R}{\On}$-invariant.
\end{proof}

\section{Semidefinite programming formulations}
\label{sec:formulation}

Before giving the semidefinite programming formulations, let us
discuss how the matrix-valued function $\Fcal_l(R)(x,y)$ can be
computed. We have
\[
  L(A_R P_\sigma) x = B^{-1} P_\sigma^{\sf T} A_R^{\sf T} x = B^{-1}
  P_\sigma^{\sf T} (A_R^{\sf T} x),
\]
where $BB^{\sf T}$ is the Cholesky decomposition of
$A_R^{\sf T} A_R = (A_R P_\sigma)^{\sf T}(A_R P_\sigma)$. This shows
that $L(A_R P_\sigma) x$ depends only on the inner products between
the vectors in the set $R \cup \{x\}$ and on the ordering of the
columns of $A_R$. Since the size of $R$ is bounded by $k-2$, this also
shows that all computations for setting up the problem can be done in
a relatively small dimension depending on~$k$ and not on~$n$.

\subsection{An SDP formulation for spherical finite-distance sets}
\label{sec:sdp-formulation}

To write the full semidefinite programming formulation corresponding
to~\eqref{prog:stepk-top-dual}, we use Corollary~\ref{cor:K-inv}
together with the isomorphism from Proposition~\ref{prop:isomorphism}.
Let $\Scal_{\succeq 0}^N$ denote the cone of $N \times N$ positive
semidefinite matrices. If for $R \in \Rcal_{k-2}$ and $0 \leq l\leq d$
we have $F_{R, l} \in \Scal_{\succeq 0}^N$, where
$N = \binom{d-l+|R|}{|R|}$, then
$T \colon S^{n-1} \times S^{n-1} \times I_{k-2} \to \R$ given by
\vspace*{-.25cm}
\[
T(x,y,Q) = \sum_{l = 0}^d \big\langle F_{\gamma_Q Q, l}, 
\Fcal_l(\gamma_QQ)(\gamma_Qx,\gamma_Qy) \big\rangle
\vspace*{-.1cm}
\]
is a function in 
$\Ccal(S^{n-1} \times S^{n-1} \times I_{k-2})_{\succeq 0}^{\On}$ and 
hence, for $S \in \Rcal_{k} \setminus \{\emptyset\}$, the expression
for $B_k T(S)$ becomes
\begin{align*}
  B_k T(S) &= \sum_{\substack{Q \subseteq S \\ |Q| \leq k-2}} \; 
  \sum_{\substack{x, y \in S \\ \{x, y\} \cup Q = S}} T(x, y, Q)\\
           &= \sum_{\substack{Q \subseteq S \\ |Q| \leq k-2}} \; 
  \sum_{\substack{x, y \in S \\ \{x, y\} \cup Q = S}} \sum_{l = 0}^d \big\langle 
  F_{\gamma_Q Q, l},\Fcal_l(\gamma_QQ)(\gamma_Qx,\gamma_Qy) 
  \big\rangle\\
           &= \sum_{\substack{Q \subseteq S \\ |Q| \leq k-2}}\; 
  \sum_{l = 0}^d \big\langle F_{\gamma_Q Q, l}, \sum_{\substack{x, y \in S\\ 
  \{x, y\} \cup Q = S}} \Fcal_l(\gamma_Q Q)(\gamma_Q x,\gamma_Q y) 
  \big\rangle.
\end{align*}

Since the action of $\On$ on $S^{n-1} \simeq I_{=1}$ is transitive,
the quotient $I_{=1} / \On$ has only one element. We set
$\mathcal R_1 \setminus \mathcal R_0 = \{e_1\}$, where $e_1$ is the
first canonical basis vector of $\R^n$. We replace the objective
$1+\lambda$ in \eqref{prog:stepk-top-dual} by $1 + B_kT(\{e_1\})$,
which we can further simplify by noticing that $Y_0^{n,1}(1,1,1)$ is
the all-ones matrix $J_{d+1}$ of size $(d+1) \times (d+1)$ and
$Y_l^{n,1}(1,1,1)$ is the zero matrix for $l > 0$.  This gives the
semidefinite programming formulation
\begin{multline*}
  \min \bigg\{\, 1 + \sum_{l = 0}^d F_{\emptyset,l} + \langle
  F_{\{e_1\},0}, J_{d+1}\rangle : F_{R, l} \in \mathcal S_{\succeq
    0}^{\binom{d-l+|R|}{|R|}}
  \text{ for } 0 \leq l\leq d \text{ and } R \in \Rcal_{k-2},\\
  \sum_{\substack{Q \subseteq S \\ |Q| \leq k-2}} \sum_{l = 0}^d
  \big\langle F_{\gamma_Q Q, l},\mkern-13mu \sum_{\substack{x,y \in S
      \\ \{x, y\} \cup Q = S}}\mkern-15mu\Fcal_l(\gamma_Q Q)(\gamma_Q
  x,\gamma_Q y)\big\rangle \leq -2 \chi_{I_{=2}}(S) \text{ for } S \in
  \mathcal R_k \setminus \mathcal R_1\,\bigg\}.
\end{multline*}
For each fixed $d$ this gives an upper bound for $\Delta_k(G)^*$ that
converges to $\Delta_k(G)^*$ as $d$ tends to infinity.

\subsection{A precise connection between the Bachoc-Vallentin bound
  and the Lasserre hierarchy}\label{sec:connection}

The bound $\Delta_2(G)^*$ immediately reduces to the generalization of
the Lov\'asz $\vartheta$ number as given by Bachoc, Nebe, Oliveira,
and Vallentin~\cite{bachoc09}, which coincides with the LP
bound~\cite{delsarte77} after symmetry reduction. Here we show that
$\Delta_3(G)^*$ can be interpreted as a nonsymmetric version of the
Bachoc-Vallentin $3$-point bound~\cite{bachoc08}.

Suppose $T$ is feasible for $\Delta_3(G)^*$. If $S = \{a, b\}$ with
$a\neq b$, then
\begin{align*}
  B_3T(\{a, b\}) &= \sum_{\substack{Q \subseteq S\\ |Q| \leq 1}} \;
  \sum_{\substack{x, y \in S \\ Q \cup \{x,y\} = S}} T(x,y,Q)\\[0.3em]
                 &= T(a,b,\emptyset) + T(b,a,\emptyset) + T(a,b,\{a\}) + T(b,a,\{a\})\\ 
                 &\qquad + T(b,b,\{a\}) + T(a,b,\{b\}) + T(b,a,\{b\}) + T(a,a,\{b\}).
\end{align*}
By using
$T(x,y,\emptyset) = \sum_{l = 0}^d F_{\emptyset,l} P_l^n(x \cdot y)$
and
\begin{align*}
  T(x, y, \{z\}) &= \sum_{l = 0}^d\big\langle F_{\{e_1\}, l}, 
                   \Fcal_l\big(\{e_1\}\big)(\gamma_{\{z\}} x, \gamma_{\{z\}}y)\big\rangle\\
                 &= \sum_{l = 0}^d\big\langle F_{\{e_1\}, l}, Y_l^{n,1}(x \cdot y, x \cdot z, y \cdot z)\big\rangle,
\end{align*}
we see that
\[
  B_3T(\{a, b\}) = 2\sum_{l = 0}^d F_{\emptyset,l} P_l^n(a \cdot b) +
  6\sum_{l = 0}^d\big\langle F_{\{e_1\}, l}, S_l^n(a\cdot b, a\cdot b,
  1)\big\rangle,
\]
where we use the notation
$S_l^n = \frac{1}{6}\sum_{\sigma \in S_3} \sigma Y_l^{n, 1}$, in which
$\sigma$ runs through the group of all permutations of three variables
and acts on $Y_l^{n, 1}$ by permuting its arguments.

If $|S| = 3$, say $S = \{a, b, c\}$, then
\begin{align*}
  B_3T(\{a, b, c\}) &= \sum_{\substack{Q \subseteq S\\ |Q| \leq 1}}\;
  \sum_{\substack{x, y \in S\\ Q \cup \{x,y\} = S}} T(x,y,Q)\\[0.4em]
                    &= T(a,b,\{c\}) + T(b,a,\{c\}) + T(a,c,\{b\})\\
                    &\qquad + T(c,a,\{b\}) + T(b,c,\{a\}) + T(c,b,\{a\})\\
                    &= 6\sum_{l = 0}^d\big\langle F_{\{e_1\}, l}, S_l^n(a \cdot b, a \cdot c, b 
                      \cdot c)\big\rangle.
\end{align*}

Using the above expressions we see that the constraints
$B_3T(S) \leq -2$ for $S \in I_{= 2}$ and $B_3T(S) \leq 0$ for
$S \in I_{= 3}$ in $\Delta_3(G)^*$ are exactly the ones that appear in
Theorem~4.2 of Bachoc and Vallentin~\cite{bachoc08}. Except for the
presence of an \textit{ad hoc} $2\times 2$ matrix variable $B$ that
comes from a separate argument, both bounds are identical.

\begin{remark}
Recall that for our method it is essential that $I_{k-2}/\On$ be
finite and that $I_{=m} / \On$ can be represented by the set of
$m \times m$ positive semidefinite matrices of rank at most $n$ with
ones in the diagonal and elements of $D$ elsewhere, up to simultaneous
permutations of the rows and columns. So $I_{k-2}/\On$ is finite for
$k=2$, $3$, but infinite whenever $D$ is infinite and $k \geq 4$. This
explains why it is not clear how to compute a $4$-point bound
generalization of the LP~\cite{delsarte77} and SDP~\cite{bachoc08}
bounds for the size of spherical codes with given minimal angular
distance. For the spherical finite-distance problem, however, the set
$I_{k-2}/\On$ is always finite, so that we can perform computations
beyond $k=3$.
\end{remark}

\section{Two-distance sets and equiangular lines} \label{sec:r-distance-sets}

If $D = \{a, -a\}$ for some $0 < a < 1$, then the vectors in a
spherical $D$-code correspond to a set of equiangular lines with
common angle $\arccos a$.  We set
\[
  M_a(n) = A\big(n, \{a,-a\}\big)
\] 
and write
\[
  M(n) = \max_{0 < a < 1} M_a(n)
\] 
for the maximum number of equiangular lines in $\R^n$ with any common
angle.

A semidefinite programming bound based on the method of Bachoc and
Vallentin~\cite{bachoc08}, and hence equivalent to $\Delta_3(G)^*$,
was applied to this problem by Barg and Yu~\cite{barg14} (see also the
table computed by King and Tang~\cite{king16}) which, together with
previous results, determines $M(n)$ for most $n \leq 43$.

Barg and Yu present~\cite[Eqs.~(14)--(17)]{barg13} a semidefinite
programming formulation that corresponds exactly to the formulation
given in Section~\ref{sec:sdp-formulation} when $k = 3$ (except for an
\textit{ad hoc} $2 \times 2$ matrix). In the other
papers~\cite{barg14, king16, okuda16, yu17} where this semidefinite
program is considered, a primal version is given instead, which is
less convenient from the perspective of rigorous verification of
bounds.

In this paper we compute new upper bounds for $M_a(n)$ for $a = 1/5$,
$1/7$, $1/9$, and~$1/11$ and many values of $n$ using $\Delta_k(G)^*$
with $k = 4$, $5$, and $6$. The results do not improve the known
bounds for $M(n)$ but greatly improve the known bounds for $M_a(n)$
for certain ranges of dimensions; these results are presented in
Section~\ref{sec:newsdp}.

\subsection{Overview of the literature}\label{sec:overview-equi}

The literature on equiangular lines is vast; here is a summary.

\subsubsection{Bounds for $M(n)$}

The interest in $M(n)$ started with Haantjes~\cite{haantjes48}, who
showed $M(3) = M(4) = 6$ in 1948. Since then, much progress has been
made using different techniques, and $M(n)$ has been determined for
many values of $n \leq 43$. Table~\ref{tab:Mn-small} presents the
known values for $M(n)$ for small dimensions.

\begin{table}[t]
\begin{adjustbox}{width=\textwidth,center}
\begin{tabular}{@{}ccccccccc@{}}
\toprule
$\bm{n}$  & $\bm{M(n)}$ & $\bm{a}$ & \textbf{SDP bound} & & $\bm{n}$ & $\bm{M(n)}$ & $\bm{a}$ & \textbf{SDP bound} 
\\\cmidrule{1-4}\cmidrule{6-9}
2    & 3   & 1/2      & 3 & & 17 & 48--49  & 1/5 & 51 \\\cmidrule{1-4}\cmidrule{6-9}
3    & 6   & $1/\sqrt{5}$ & 6 & & 18 & 56--60 & 1/5 & 61 \\\cmidrule{1-4}\cmidrule{6-9}
4    & 6   & 1/3, $1/\sqrt{5}$ & 6 & & 19 & 72--75 & 1/5 & 76 \\\cmidrule{1-4}\cmidrule{6-9}
5    & 10  & 1/3      & 10 & & 20 & 90--95  & 1/5 & 96 \\\cmidrule{1-4}\cmidrule{6-9}
6    & 16  & 1/3      & 16 & & 21 & 126    & 1/5 & 126 \\\cmidrule{1-4}\cmidrule{6-9}
7--13 & 28  & 1/3      & 28 & & 22 & 176    & 1/5 & 176 \\\cmidrule{1-4}\cmidrule{6-9}
14 & 28--29 & 1/3, 1/5 & 30 & & 23--41 & 276 & 1/5 & 276 \\\cmidrule{1-4}\cmidrule{6-9}
15   & 36  & 1/5      & 36 & & 42 & 276--288 & 1/5, 1/7 & 288 \\\cmidrule{1-4}\cmidrule{6-9}
16 & 40--41 & 1/5      & 42 & & 43 & 344    & 1/7 & 344\\\bottomrule
\end{tabular}
\end{adjustbox}\bigskip
\caption{Known values for $M(n)$ for small dimensions together with
the cosine $a$ of the common angle between the lines. The values
known exactly were determined by several authors \cite{barg14,
haantjes48, lemmens73, lint66}. Most lower bounds are collected by
Lemmens and Seidel~\cite{lemmens73}, except for dimensions $18$,
$19$, and $20$ \cite{linyu18}, \cite[p.123]{taylor71}. The
remaining upper bounds \cite{greaves16, greaves18a, greaves18b,
yu15} do not rely on semidefinite programming.}
\label{tab:Mn-small}
\end{table}

The most general bound for $M(n)$, called the \emph{absolute bound},
is due to Gerzon:

\begin{theorem}[Gerzon, cf.\ Lemmens and Seidel~\cite{lemmens73}]\label{thm:gerzon}
We have
\begin{equation*}
M(n) \leq \frac{n(n+1)}{2}.
\end{equation*}
Moreover, if equality holds, then $n = 2$, $n=3$, or $n = l^2 - 2$ for
some odd integer~$l$ and the cosine of the common angle is $a = 1/l$.
\end{theorem}

The four cases where it is known that the bound is attained are
$n = 2$, $3$, $7$, and~$23$. Delsarte, Goethals and
Seidel~\cite[Example 8.3]{delsarte77} show that equality holds if and
only if the union of the code with its antipodal code is a tight
spherical 5-design, and in this case Cohn and Kumar~\cite{cohn07} show
this union is a universally optimal code (which means it minimizes
every completely monotonic potential function in the squared chordal
distance). Bannai, Munemasa, and Venkov~\cite{bannai04} and Nebe and
Venkov~\cite{nebe12} show that there are infinitely many odd
integers~$l$ for which no tight spherical 5-design exists in $S^{n-1}$
with $n = l^2 - 2$, so that Gerzon's bound cannot be attained in those
dimensions. This list starts with $l = 7$, $9$, $13$, $21$, $25$,
$45$, $57$, $61$, $69$, $85$, $93$, \dots\ (resp.\ $n = 47$, $79$,
$167$, $439$, $623$, $2023$, $3247$, $3719$, $4759$, $7223$, $8647$,
\dots). For the remaining possible dimensions, attainability is an
open problem.

For the dimensions that are not of the form $l^2 - 2$ for some odd
integer $l$, the absolute bound can be improved:

\begin{theorem}[Glazyrin and Yu~\cite{glazyrin18} and King and 
Tang~\cite{king16}]\label{thm:glazyrin1}
Let $l$ be the unique odd integer such that
$l^2-2 \leq n \leq (l+2)^2-3$. Then,
\begin{equation*}
M(n) \leq 
\begin{cases}
\displaystyle\frac{n(l+1)(l+3)}{(l+2)^2-n}, & 
n = 44,45,46,76,77,78,117,118,166,222,286,358;\\[1em]
\displaystyle\frac{(l^2-2)(l^2-1)}{2}, &  
\text{for all other } n \geq 44.
\end{cases}
\end{equation*}
Furthermore, if the bound is attained, then the cosine of the angle
between the lines is $a = 1/(l+2)$ for the first case and $a = 1/l$
for the second.
\end{theorem}

Glazyrin and Yu also proved another theorem~\cite[Theorem
4]{glazyrin18} about the codes that attain the bound from
Theorem~\ref{thm:glazyrin1}:

\begin{theorem}[Glazyrin and Yu~\cite{glazyrin18}]\label{thm:glazyrin2}
  Suppose $l$ is a positive odd integer. If $X$ is a
  $\{1/l,-1/l\}$-spherical code of size $(l^2-2)(l^2-1)/2$ contained
  in $S^{n-1}$ with $n \leq 3l^2-16$, then $X$ must belong to a
  $(l^2-2)$-dimensional subspace.
\end{theorem}

Since $(l+2)^2-3 \leq 3l^2-16$ for $l \geq 5$, this last theorem
implies that if the second bound from Theorem~\ref{thm:glazyrin1} is
attained, then Gerzon's bound also has to be attained for $n =
l^2-2$. For the first two cases where tight spherical 5-designs do not
exist, this implies $M(n) \leq 1127$ for $47 \leq n \leq 75$ and
$M(n) \leq 3159$ for $79 \leq n \leq 116$.
The following theorem is adapted from Larman, Rogers, and
Seidel~\cite[Theorem~2]{larman77}:

\begin{theorem}[Larman, Rogers, and Seidel~\cite{larman77}]\label{thm:larman}
We have
\begin{equation*}
  M(n) \leq \max\{2n+3, M_{1/3}(n), M_{1/5}(n), \ldots, M_{1/l}(n)\},
\end{equation*}
where $l$ is the largest odd integer such that $l \leq \sqrt{2n}$.
\end{theorem}

Most of the results for $M(n)$ rely on Theorem~\ref{thm:larman}, which
shows that to bound $M(n)$ one just has to consider finitely many
angles.  This motivates the consideration of $M_a(n)$ when $1/a$ is an
odd integer.

\subsubsection{Bounds for $M_a(n)$}

Bounds for fixed $a$ are known as relative bounds, as opposed to
Gerzon's absolute bound from Theorem~\ref{thm:gerzon}. The first
relative bound is due to van Lint and Seidel~\cite{lint66}:

\begin{theorem}[van Lint and Seidel~\cite{lint66}]\label{thm:lint}
If $n < 1/a^2$, then
\begin{equation*}
M_a(n) \leq \frac{n(1-a^2)}{1 - na^2}.
\end{equation*}
\end{theorem}

As shown by Glazyrin and Yu~\cite[Theorem 5]{glazyrin18},
Theorem~\ref{thm:lint} can be derived from the positivity of the
Gegenbauer polynomials $P_2^n$, and indeed this is the bound given by
the semidefinite programming techniques when $n \leq 1/a^2 - 2$. This
bound is also the first case of Theorem~\ref{thm:glazyrin1}.

After computing the semidefinite programming bound for many values of
$n$ and~$a$, Barg and Yu~\cite{barg14} observed long ranges
$1/a^2-2\leq n \leq 3/a^2 - 16$ where the bound remained stable,
matching Gerzon's bound (Theorem~\ref{thm:gerzon}) at $n = 1/a^2 -
2$. Based on this observation, Yu~\cite{yu17} proved the following
theorem:
\begin{theorem}[Yu~\cite{yu17}]\label{thm:yu}
If $n \leq 3/a^2 - 16$ and $a \leq 1/3$, then
\begin{equation*}
 M_a(n) \leq \frac{(1/a^2-2)(1/a^2-1)}{2}.
\end{equation*}
\end{theorem}

An alternative proof for the previous theorem is given by Glazyrin and
Yu~\cite[Theorem 6]{glazyrin18}, where the use of the positivity of
the Gegenbauer polynomials $P_1^{n-1}$ and $P_3^{n-1}$ is made more
explicit. The bounds given by the semidefinite programming method were 
also used to prove the following theorem:

\begin{theorem}[Okuda and Yu~\cite{okuda16}]\label{thm:okuda}
If $3/a^2 - 16 \leq n \leq 3/a^2+6/a+1$, then
\begin{equation*}
 M_a(n) \leq 2 + \frac{(n-2)(1/a+1)^3}{(3/a^2-6/a+2)-n}.
\end{equation*}
\end{theorem}

The bounds from Theorems~\ref{thm:lint},~\ref{thm:yu},
and~\ref{thm:okuda} coincide with the values given by the semidefinite
programming formulation when $k = 3$ (see the points labeled
``$\Delta_3(G)^*$~\cite{barg14, king16}'' in Figures
\ref{fig:plot5}--\ref{fig:plot11}).
Another source of relative bounds is a technique called pillar
decomposition, introduced by Lemmens and Seidel~\cite{lemmens73} and
used to determine $M_{1/3}(n)$:

\begin{theorem}[Lemmens and Seidel~\cite{lemmens73}]\label{thm:lemmens}
If $n \geq 15$, then
\begin{equation*}
M_{1/3}(n) = 2n - 2.
\end{equation*}
\end{theorem}
For $a = 1/5$, they obtained results that lead to the following
conjecture:

\begin{conjecture}[Lemmens and
  Seidel~\cite{lemmens73}]\label{conj:seidel}
We have
\begin{equation*}
M_{1/5}(n) = 
\begin{cases}
276 &   \text{for } 23\leq n\leq  185;\\
\lfloor\frac{3}{2}(n-1)\rfloor &  \text{for } n \geq 185.
\end{cases}
\end{equation*}
\end{conjecture}
Note that $276$ is the bound given by Theorem~\ref{thm:yu} when
$a = 1/5$ and this shows (together with the fact that there exists a
$\{-1/5, 1/5\}$-code of size $276$ in dimension $n = 23$) that the
conjecture is true for $n \leq 59$.  In fact, the semidefinite
programming bound computed by Barg and Yu~\cite{barg14} also shows
$M_{1/5}(60) = 276$. Neumaier~\cite{neuimaier89} (see
also~\cite[Corollary 6.6]{greaves16}) proved that there exists a large
$N$ such that $M_{1/5}(n) = \lfloor\frac{3}{2}(n-1)\rfloor$ for all
$n > N$.  Neumaier claimed, without a proof, that $N$ should be at
most $30251$.

Recently, Lin and Yu~\cite{lin19} made progress in this conjecture by
proving some claims from Lemmens and Seidel~\cite{lemmens73}. The only
case still open is when the code has a set with $4$ unit vectors with
mutual inner products $-1/5$ and no such set with~$5$ unit vectors (up
to replacement of some vectors by their antipodes).

Glazyrin and Yu~\cite{glazyrin18} introduced a new method to derive
upper bounds for spherical finite-distance sets. By using Gegenbauer
polynomials together with the polynomial method, they proved a theorem
that, specialized for two-distance sets, is:

\begin{theorem}[Glazyrin and Yu~\cite{glazyrin18}]\label{thm:glazyrin3}
 For all~$a$, $b$, and~$n$, we have
\begin{equation*}
 A\big(n, \{a,b\}\big) \leq \frac{n+2}{1 - (n-1)/(n(1-a)(1-b))}
\end{equation*}
whenever the right-hand side is positive.
\end{theorem}

With this result, they proved the following relative bound, which
provides the best bounds for moderately large values of $n$ (see
Figures~\ref{fig:plot7}--\ref{fig:plot11}):

\begin{theorem}[Glazyrin and Yu~\cite{glazyrin18}]\label{thm:glazyrin4}
If $a \leq 1/3$, then
\begin{equation*}
\begin{aligned}
M_a(n) &\leq n\bigg(\frac{(a^{-1}-1)(a^{-1}+2)^2}{3a^{-1}+5} + 
\frac{(a^{-1}+1)(a^{-1}-2)^2}{3a-5}+2\bigg)+2\\ 
&\leq n\bigg(\frac{2}{3}a^{-2}+\frac{4}{7}\bigg)+2.
\end{aligned}
\end{equation*}
\end{theorem}

King and Tang~\cite{king16} improved the pillar decomposition
technique and got a better bound for $M_{1/5}(n)$~\cite[Theorem
7]{king16}. Recently, Lin and Yu~\cite{lin19} further improved parts
of their argument; by combining~\cite[Proposition 4.5]{lin19} with the
proof of~\cite[Theorem 7]{king16} we get:

\begin{theorem}[Lin and Yu~\cite{lin19}]\label{thm:lin1}
If $n \geq 63$, then
\begin{equation*}
M_{1/5}(n) \leq 100 + 3 A(n-4, \{1/13, -5/13\}).
\end{equation*}
\end{theorem}

The previous results give three competing methods to bound
$M_{1/5}(n)$, each one being the best for a different range of
dimensions. One can either use semidefinite programming to bound
$M_{1/5}(n)$ directly, use Theorem~\ref{thm:lin1} together with
semidefinite programming to bound $A\big(n-4, \{1/13, -5/13\}\big)$,
or use Theorem~\ref{thm:glazyrin3}.  King and Tang~\cite{king16} made
this comparison, computing the semidefinite programming
bound~$\Delta_3(G)^*$. See in Table~\ref{table:values-a5} and in
Figure~\ref{fig:plot5} the comparison with the new semidefinite
programming bound $\Delta_6(G)^*$.


Regarding asymptotic results, while it is known that $M(n)$ is
asymptotically quadratic in $n$, for fixed $a$ we have that $M_a(n)$
is linear in $n$. Bukh~\cite{bukh16} was the first to show a bound for
$M_a(n)$ of the form $M_a(n) \leq c n$, although with a large
constant~$c$. Theorem~\ref{thm:glazyrin4} has another linear bound good to give results for intermediate values of $n$, while the best asymptotic result is due to Jiang et al.~\cite{jiang19}. They completely settled the value of $\lim_{n \to \infty} M_a(n) / n$ for every $a$ in terms of a parameter called the \emph{spectral radius order} $r(\lambda)$, which is defined for $\lambda > 0$ as the smallest integer $r$ so there exists a graph with $r$ vertices and adjacency matrix with largest eigenvalue exactly~$\lambda$, and is defined $r(\lambda) = \infty$ in case no such graph exists.

\begin{theorem}[Jiang et al.~\cite{jiang19}]\label{thm:jiang}
Fix $0 < a < 1$. Let $\lambda = (1-a)/(2a)$ and $r = r(\lambda)$ be its spectral radius order. The maximum number $M_a(n)$ of equiangular lines in $\R^n$ with common angle $\arccos a$ satisfies
\begin{enumerate}[(a)]
 \item $M_a(n) = \lfloor r(n-1)/(r-1) \rfloor$ for all sufficiently large $n > n_0(a)$ if $r < \infty$.
 
 \item $M_a(n) = n + o(n)$ as $n \to \infty$ if $r = \infty$.
\end{enumerate}
\end{theorem}

Jiang et al. remarks that the $n_0(a)$ from their theorem may be really big, though. When $a = 1/(2r-1)$ for some positive integer $r$, then $\lambda = r-1$ and $r(\lambda) = r$ (since the complete graph on $r$ vertices has spectral radius $r-1$). Theorem~\ref{thm:jiang} confirms a conjecture made by Bukh~\cite{bukh16}:

\begin{corollary}[Jiang et al.~\cite{jiang19}]\label{conj:bukh}
If $a = 1/(2r-1)$ for some positive integer $r \geq 2$, then for all $n$ sufficiently large,
\begin{equation*}
 M_{a}(n) = \left\lfloor \frac{r(n-1)}{r-1} \right\rfloor .
\end{equation*} 
\end{corollary}

There is a simple construction that achieves the value from Corollary~\ref{conj:bukh}. Let $a = 1/(2r-1)$ for some positive integer $r$ and $t, s$ be arbitrary positive integers. Then one can show that a matrix with~$t$ diagonal blocks, each of size $r$, and $s$ diagonal blocks of size $1$, with diagonal entries equal to~1, off-diagonal entries inside each block equal to~$-a$, and all other entries equal to~$a$ is the Gram matrix of a $\{-a,a\}$-code in $S^{(r-1)t+s}$ of size $rt+s$. Letting $t = \lfloor (n-1)/(r-1) \rfloor$ and $s = n-1 - (r-1)\lfloor (n-1)/(r-1) \rfloor$ we get the desired size.


\subsection{New semidefinite programming bounds}
\label{sec:newsdp}

As observed in Section~\ref{sec:r-distance-sets}, the semidefinite
programming bounds computed by Barg and Yu~\cite{barg14} and King and
Tang~\cite{king16} correspond to $\Delta_3(G)^*$. In this paper we
compute new upper bounds for $M_a(n)$ for $a = 1/5$, $1/7$, $1/9$,
and~$1/11$ and many values of $n$ using $k = 4$, $5$, and~$6$. We
always use degree $d = 5$ for the polynomials since, as reported by
Barg and Yu~\cite{barg13}, no improvement is observed for larger
values of $d$ (but this may change if sets~$D$ with cardinality
greater than~$2$ are considered). The semidefinite programs were
produced using a script written in Julia~\cite{bezanson17} using
  Nemo~\cite{fieker17}, were solved with SDPA-GMP~\cite{nakata10},
and the results were rigorously verified using the interval arithmetic
library Arb~\cite{johansson17}. The rigorous verification procedure is
much simpler than that for similar problems~\cite{dostert17}. The
scripts used to generate the programs and verify the results can be
found with the arXiv version of this paper.

The results are presented in Figures \ref{fig:plot5}--\ref{fig:plot11}
and Tables~\ref{table:values-a5}--\ref{table:values-a11}, where we
compile the bounds for $M_a(n)$ for each~$n$ that is a multiple
of~$5$; the best bounds are displayed in boldface.  While it takes
only a few seconds to generate and solve a single instance of the
semidefinite programming problem for~$k = 3$, the process takes
about~5 days using a single core of an Intel i7-8650U processor 
for~$k = 6$; that is why the tables have some missing values 
for~$\Delta_6(G)^*$.

No improvements were obtained for $n \leq 3/a^2 -16$; we observed in
this case that $\Delta_6(G)^* = \Delta_3(G)^*$ which is equal to the
values given by Theorems~\ref{thm:lint} and~\ref{thm:yu}. Since this
is the range of dimensions that influences $M(n)$, no improvements for
$M(n)$ were obtained.
We obtained great improvements for all dimensions $n > 3/a^2 -16$,
making the semidefinite programming bound competitive with the other
methods (like Theorem~\ref{thm:glazyrin4}) for more
dimensions. Asymptotically, the semidefinite programming bounds behave
badly, loosing even to Gerzon's bound.

In particular, we improved the range of dimensions for which the bound
remains stable, showing that $n = 3/a^2 -16$ from Theorem~\ref{thm:yu}
is not optimal.  Table~\ref{table:new-ranges} shows how much this
range is increased for the values of $a$ considered. This observation
motivates the following two questions, where~$a$ is such that~$1/a$ is
an odd integer:

\begin{enumerate}
\item What is the smallest~$n$ such that
  $M_a(n) = (1/a^2-2)(1/a^2-1)/2$?

\item What is the smallest~$n$ such that
  $M_a(n) > (1/a^2-2)(1/a^2-1)/2$?
\end{enumerate}

\begin{table}[t]
\begin{center}{ 
\begin{tabular}{@{}cccccc@{}}
\toprule
$\bm{a}$ & $\bm{(1/a^2-2)(1/a^2-1)/2}$ & $\bm{\Delta_3(G)^*}$ {\footnotesize\cite{barg14, king16}} & $\bm{\Delta_4(G)^*}$ & 
$\bm{\Delta_5(G)^*}$ & $\bm{\Delta_6(G)^*}$ \\\midrule
$1/5$ & $276$ & $60$ & $65$ & $69$ & $70$ \\\midrule
$1/7$ & $1128$ & $131$ & $145$ & $158$ & $169$ \\\midrule
$1/9$ & $3160$ & $227$ & $251$ & $273$ & $300$ \\\midrule
$1/11$ & $7140$ & $347$ & $381$ & $413$ & $448$\\\bottomrule
\end{tabular}}
\end{center}\bigskip
\caption{By considering~$\Delta_k(G)^*$ for $k \geq 4$ we find out
  that the maximum dimension~$n$ for which the bound
  $M_a(n) \leq (1/a^2-\nobreak 2)(1/a^2-1)/2$ is valid is larger
  than~$3/a^2-16$ as given by Theorem~\ref{thm:yu}
  and~$\Delta_3(G)^*$; the table shows the improved dimensions.}
\label{table:new-ranges}
\end{table}

Question~(1) is the more interesting of the two since if the
smallest~$n$ is~$1/a^2 - 2$, then Gerzon's bound is
attained. Theorem~\ref{thm:glazyrin2} makes progress in this
direction, showing that Gerzon's bound is also attained if the
smallest~$n$ is at most~$3/a^2-16$; this is known not to be the case
for many $a$ (due to the nonexistence of some tight spherical
5-designs, as mentioned after Theorem~\ref{thm:gerzon}), which implies
$M_{1/7}(n) \leq 1127$ for $n \leq 131$ and $M_{1/9}(n) \leq 3159$ for
$n \leq 227$.  Table~\ref{table:new-ranges} also suggests that the
constraint $n \leq 3/a^2-16$ in Theorem~\ref{thm:glazyrin2} may not be
optimal.

Question~(2) seems interesting because Table~\ref{table:new-ranges}
shows that $n = 3/a^2 - 15$ is not a good candidate solution. In fact,
the smallest~$n$ is likely much larger, as suggested by
Conjecture~\ref{conj:seidel} for~$M_{1/5}(n)$ and the construction
described after Corollary~\ref{conj:bukh}. Using this construction,
we know that $(1/a^2-2)(1/a^2-1)/2$ is achieved when
$n = (1/a^2-2)(1/a-1)^2/2+1$, which corresponds to the dimensions
$185$, $847$, $2529$, and $5951$ for $a = 1/5$, $1/7$, $1/9$, and
$1/11$ respectively.

We also improve the bounds computed by King and Tang~\cite{king16} for
$M_{1/5}(n)$ by replacing their theorem~\cite[Theorem 7]{king16} by
Theorem~\ref{thm:lin1} and by using $\Delta_6(G)^*$ to compute better
bounds for $A(n, \{1/13, -5/13\})$. Lin and Yu~\cite{lin19} observed
that $A(n, \{1/13, -5/13\}) \geq 3n/2-3$ and therefore there is a
limit to the power of this approach: it will never be able to prove
Conjecture~\ref{conj:seidel} no matter how much we increase $k$. In
general, it is not clear how good the bound $\Delta_k(G)^*$ can
be for $M_{a}(n)$ if one allows $k$ to increase; in contrast, de Laat
and Vallentin~\cite[Theorem 2]{laat15} show that their version of the
Lasserre hierarchy for compact topological packing graphs converges to
the independence number if enough steps are computed. Whether such a
convergence result holds for $\Delta_k(G)^*$ is an open question; in
any case, it takes days to compute~$\Delta_k(G)^*$ for~$k = 6$, so one
can expect that solving the resulting semidefinite programs
for~$k > 6$ will be hard in practice.

\begin{figure}[p]
\centering
\includegraphics[width=\textwidth]{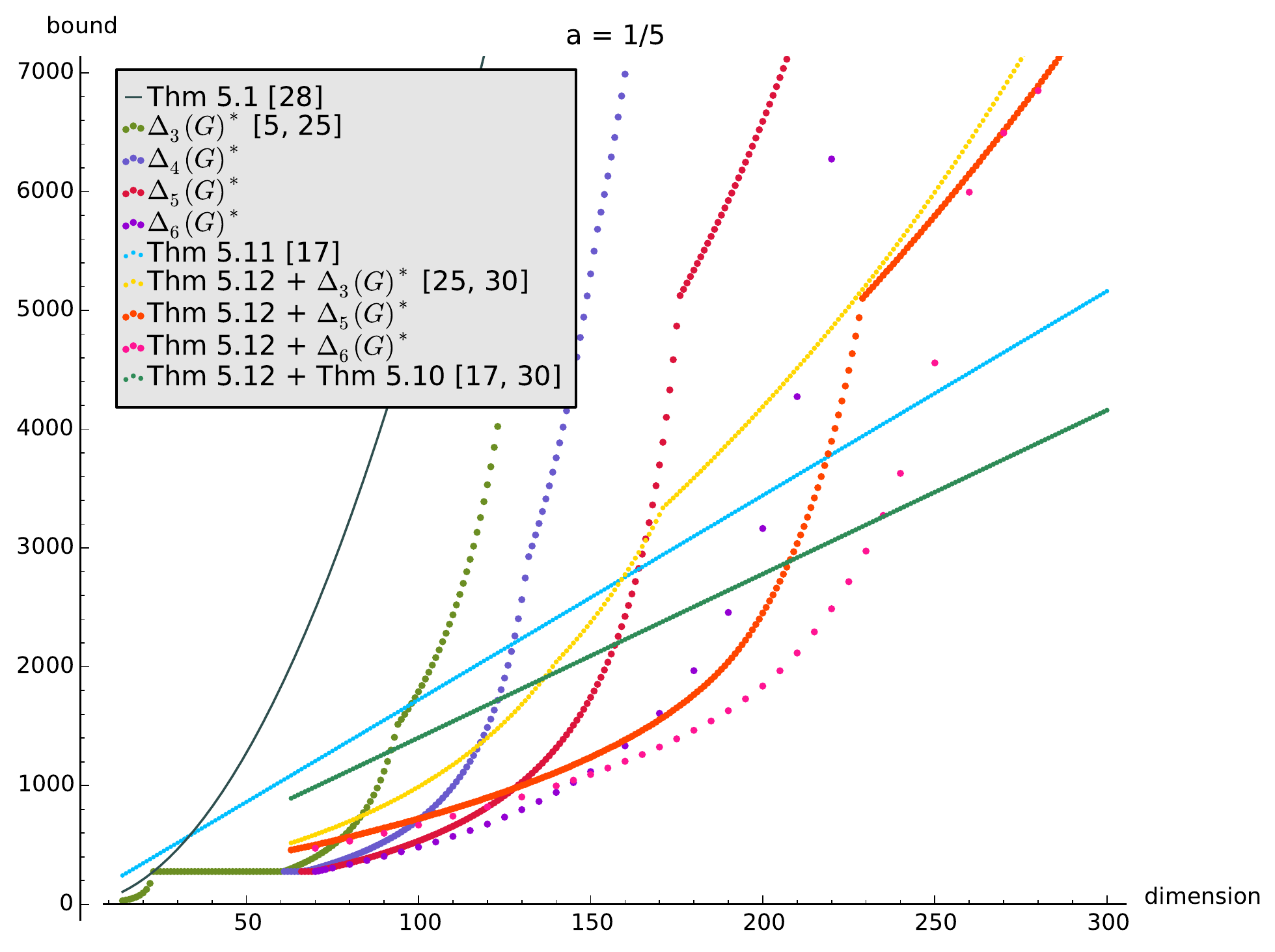}
\caption{Relative bounds for $M_{1/5}(n)$. In fact, King and
  Tang~\cite{king16} computed a bound using $\Delta_3(G)^*$ together
  with a theorem~\cite[Theorem 7]{king16}\label{fig:plot5} weaker than
  Theorem~\ref{thm:lin1}; the result is similar though.}
\end{figure}

\begin{figure}[p]
\centering
\includegraphics[width=\textwidth]{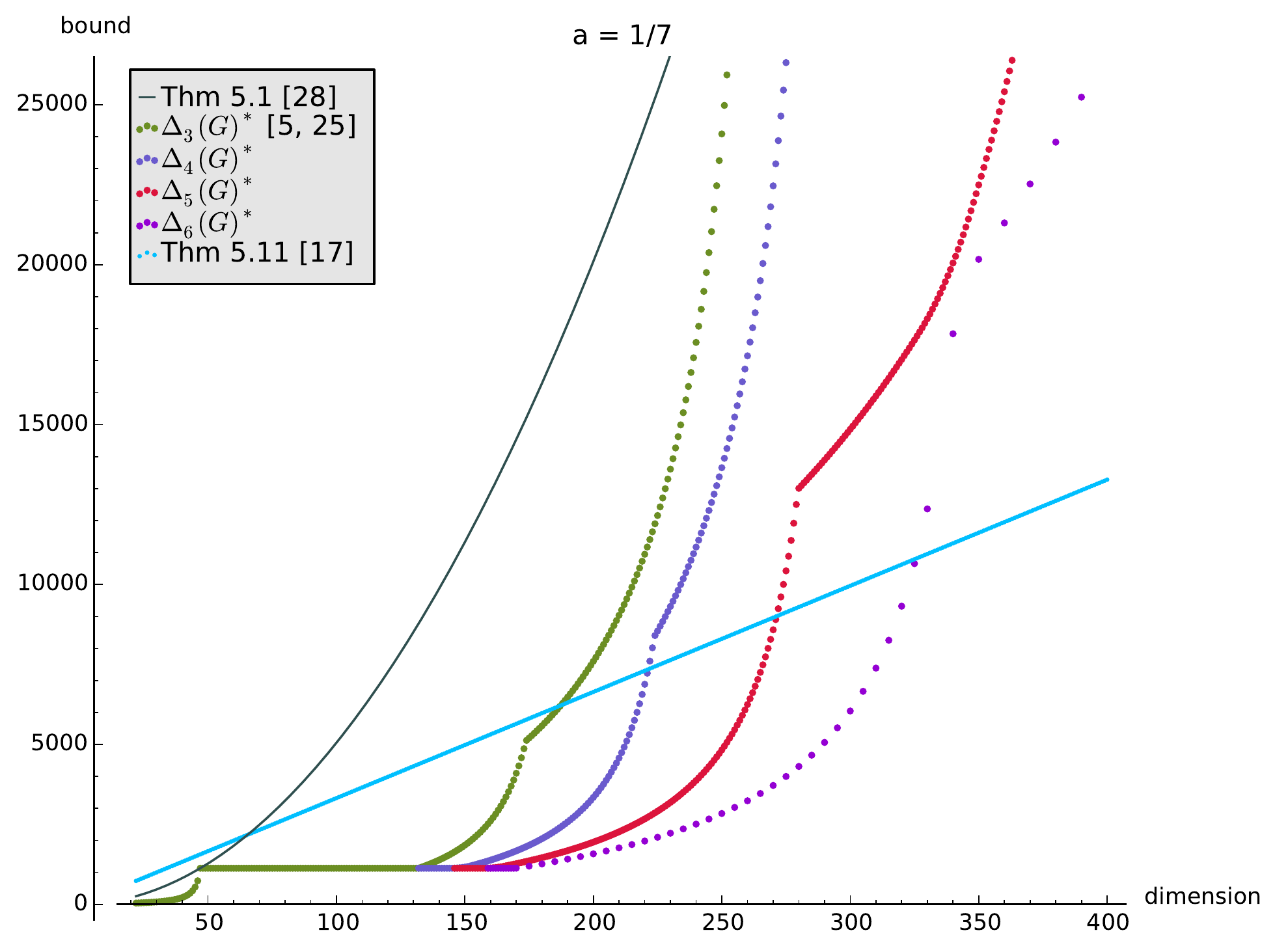}
\caption{Relative bounds for $M_{1/7}(n)$.}\label{fig:plot7}
\end{figure}

\begin{figure}[p]
\centering
\includegraphics[width=\textwidth]{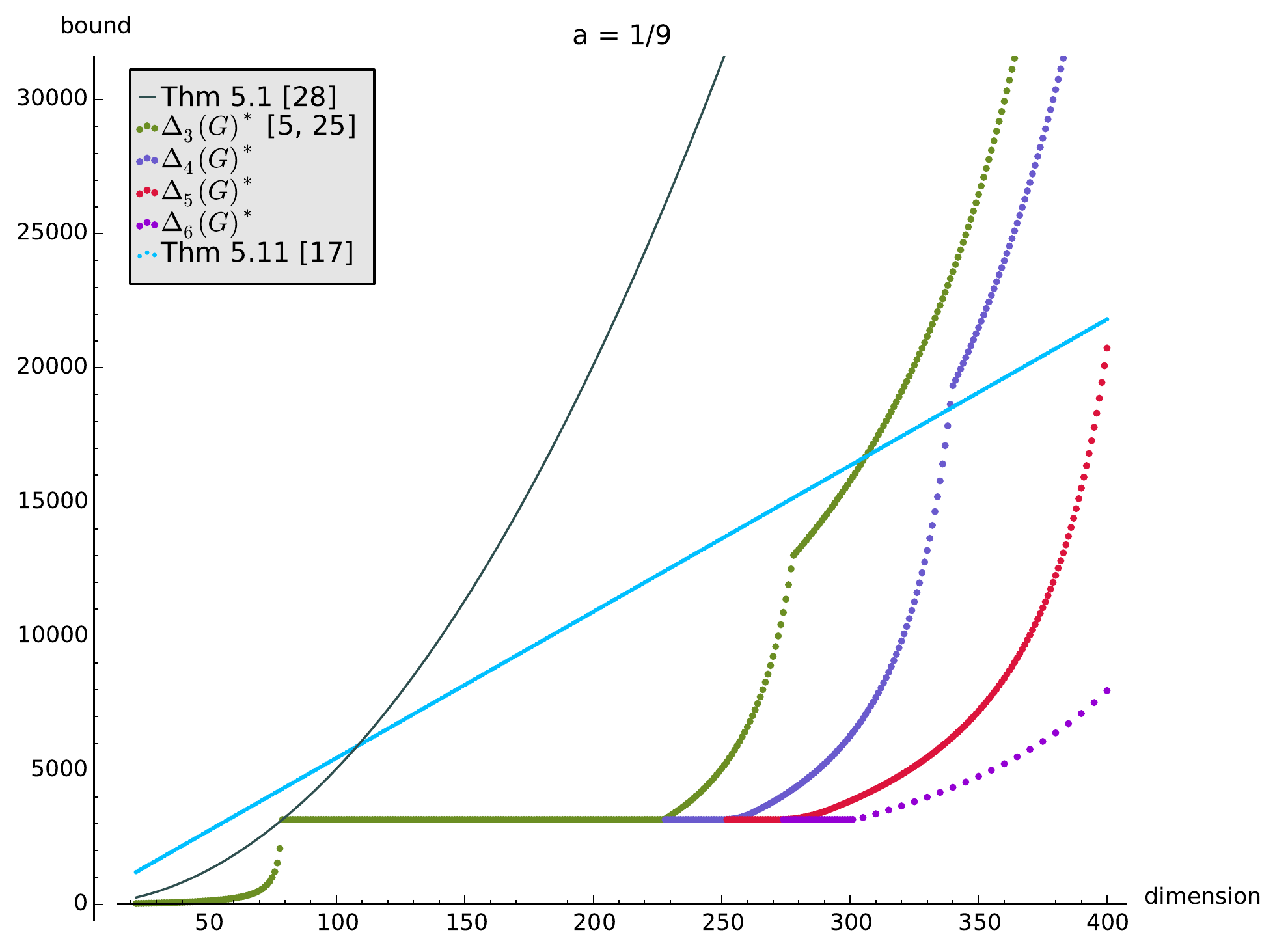}
\caption{Relative bounds for $M_{1/9}(n)$.}\label{fig:plot9}
\end{figure}

\begin{figure}[p]
\centering
\includegraphics[width=\textwidth]{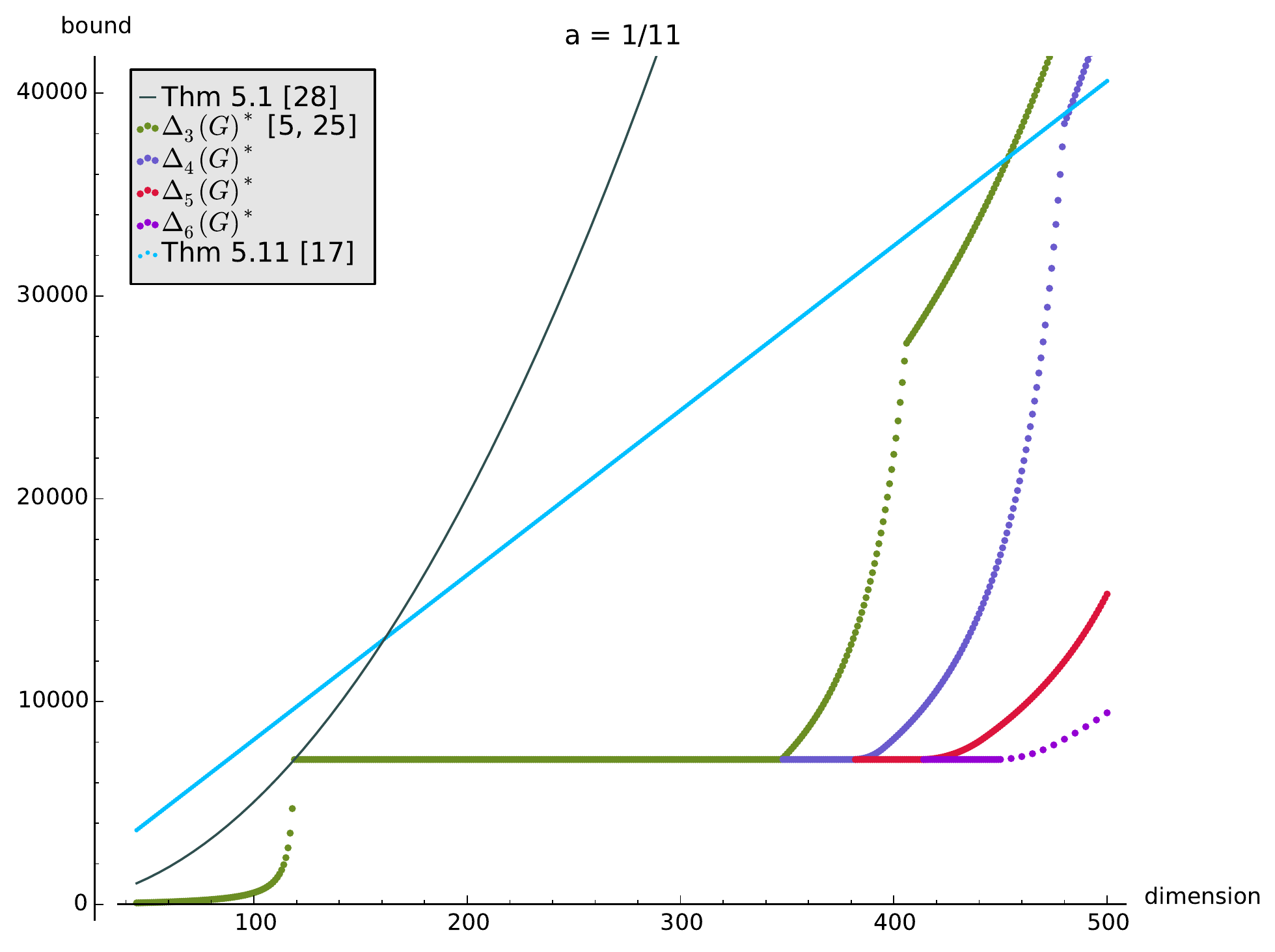}
\caption{Relative bounds for $M_{1/11}(n)$.}\label{fig:plot11}
\end{figure}

\begin{table}[p]\small
\begin{adjustbox}{width=1\textwidth,center}
\begin{tabular}{@{}cccccccc@{}}
\toprule
$\bm{n}$ & $\bm{\Delta_3(G)^*}$~\cite{barg14, king16} & $\bm{\Delta_4(G)^*}$ & 
$\bm{\Delta_5(G)^*}$ & $\bm{\Delta_6(G)^*}$ & 
\textbf{Thm~\ref{thm:lin1}~\cite{lin19}} & 
\textbf{Thm~\ref{thm:lin1}~\cite{lin19}} & 
\textbf{Thm~\ref{thm:lin1}~\cite{lin19}} \\
 & & & & & 
\textbf{ + } $\bm{\Delta_5(G)^*}$ & 
\textbf{ + } $\bm{\Delta_6(G)^*}$ & 
\textbf{+ Thm~\ref{thm:glazyrin3}~\cite{glazyrin18}} 
\\\midrule
$60$ & $\bm{276}$ & $276$ & $276$ & $276$ &  &  &  \\\midrule
$65$ & $326$ & $\bm{276}$ & $276$ & $276$ & $469$ &  & $920$  \\\midrule
$70$ & $398$ & $301$ & $278$ & $\bm{276}$ & $499$ & $472$ & $989$  \\\midrule
$75$ & $494$ & $346$ & $312$ & $\bm{305}$ & $532$ &  & $1057$  \\\midrule
$80$ & $626$ & $397$ & $348$ & $\bm{336}$ & $568$ & $532$ & $1126$  \\\midrule
$85$ & $816$ & $456$ & $388$ & $\bm{369}$ & $604$ &  & $1195$  \\\midrule
$90$ & $1120$ & $526$ & $431$ & $\bm{404}$ & $643$ & $598$ & $1264$  \\\midrule
$95$ & $1556$ & $609$ & $479$ & $\bm{442}$ & $679$ &  & $1333$  \\\midrule
$100$ & $1790$ & $710$ & $532$ & $\bm{482}$ & $721$ & $667$ & $1402$  \\\midrule
$105$ & $2077$ & $836$ & $591$ & $\bm{525}$ & $763$ &  & $1471$  \\\midrule
$110$ & $2437$ & $994$ & $657$ & $\bm{572}$ & $805$ & $742$ & $1540$  \\\midrule
$115$ & $2904$ & $1203$ & $732$ & $\bm{621}$ & $850$ &  & $1609$  \\\midrule
$120$ & $3532$ & $1489$ & $817$ & $\bm{675}$ & $898$ & $820$ & $1677$  \\\midrule
$125$ & $4419$ & $1905$ & $915$ & $\bm{734}$ & $946$ &  & $1746$  \\\midrule
$130$ & $5770$ & $2565$ & $1028$ & $\bm{797}$ & $1000$ & $904$ & $1815$  \\\midrule
$135$ & $8076$ & $3206$ & $1160$ & $\bm{866}$ & $1054$ &  & $1884$  \\\midrule
$140$ & $12896$ & $3759$ & $1317$ & $\bm{942}$ & $1111$ & $997$ & $1953$  \\\midrule
$145$ & $29280$ & $4450$ & $1508$ & $\bm{1025}$ & $1174$ & $1045$ & $2022$  \\\midrule
$150$ &  & $5307$ & $1742$ & $1117$ & $1237$ & $\bm{1093}$ & $2091$  \\\midrule
$155$ &  & $6131$ & $2038$ &  & $1309$ & $\bm{1147}$ & $2160$  \\\midrule
$160$ &  & $6989$ & $2424$ & $1334$ & $1384$ & $\bm{1204}$ & $2229$  \\\midrule
$165$ &  & $8005$ & $2948$ &  & $1465$ & $\bm{1261}$ & $2298$  \\\midrule
$170$ &  & $9166$ & $3699$ & $1608$ & $1555$ & $\bm{1324}$ & $2367$  \\\midrule
$175$ &  & $10401$ & $4868$ &  & $1654$ & $\bm{1393}$ & $2436$  \\\midrule
$180$ &  & $11833$ & $5339$ & $1967$ & $1765$ & $\bm{1465}$ & $2504$  \\\midrule
$185$ &  & $13552$ & $5623$ &  & $1891$ & $\bm{1543}$ & $2573$  \\\midrule
$190$ &  & $15647$ & $5925$ & $2457$ & $2041$ & $\bm{1630}$ & $2642$  \\\midrule
$195$ &  & $18244$ & $6247$ &  & $2224$ & $\bm{1729}$ & $2711$  \\\midrule
$200$ &  & $21531$ & $6591$ & $3164$ & $2452$ & $\bm{1837}$ & $2780$  \\\midrule
$205$ &  & $25795$ & $6960$ &  & $2719$ & $\bm{1966}$ & $2849$  \\\midrule
$210$ &  & $31508$ & $7356$ & $4274$ & $3037$ & $\bm{2116}$ & $2918$  \\\midrule
$215$ &  & $39487$ & $7783$ &  & $3421$ & $\bm{2293}$ & $2987$  \\\midrule
$220$ &  & $51276$ & $8243$ & $6274$ & $3898$ & $\bm{2488}$ & $3056$  \\\midrule
$225$ &  & $70170$ & $8741$ &  & $4495$ & $\bm{2716}$ & $3125$  \\\midrule
$230$ &  & $104611$ & $9281$ & $8667$ & $5134$ & $\bm{2974}$ & $3194$  \\\midrule
$235$ &  &  & $9870$ &  & $5296$ & $3274$ & $\bm{3263}$  \\\midrule
$240$ &  &  & $10514$ & $9407$ & $5458$ & $3628$ & $\bm{3332}$  \\\midrule
$245$ &  &  & $11221$ &  & $5626$ &  & $\bm{3401}$  \\\midrule
$250$ &  &  & $12001$ & $10226$ & $5797$ & $4558$ & $\bm{3469}$  \\\midrule
$255$ &  &  & $12865$ &  & $5971$ &  & $\bm{3538}$  \\\midrule
$260$ &  &  & $13828$ & $11137$ & $6148$ & $5995$ & $\bm{3607}$  \\\midrule
$265$ &  &  & $14908$ &  & $6328$ &  & $\bm{3676}$  \\\midrule
$270$ &  &  & $16128$ & $12155$ & $6511$ & $6493$ & $\bm{3745}$  \\\midrule
$275$ &  &  & $17516$ &  & $6700$ &  & $\bm{3814}$  \\\midrule
$280$ &  &  & $19122$ & $13302$ & $6889$ & $6850$ & $\bm{3883}$  \\\midrule
$285$ &  &  & $21199$ &  & $7084$ &  & $\bm{3952}$  \\\midrule
$290$ &  &  & $23982$ & $14601$ & $7285$ & $7219$ & $\bm{4021}$  \\\midrule
$295$ &  &  & $27058$ &  & $7489$ &  & $\bm{4090}$  \\\midrule
$300$ &  &  & $30840$ & $16086$ & $7696$ & $7600$ & $\bm{4159}$  \\\bottomrule
\end{tabular}
\end{adjustbox}\bigskip
\caption{Upper bounds for $M_{1/5}(n)$ by diverse methods, including new results 
with $\Delta_6(G)^*$. The best bound in each dimension is in boldface.} 
\label{table:values-a5}
\end{table}

\begin{table}[p]\small
\begin{adjustbox}{width=1.25\textwidth,center}
\begin{tabular}{@{}ccccccccccccc@{}}
\toprule
$\bm{n}$ & $\bm{\Delta_3(G)^*}$~\cite{barg14, king16} & $\bm{\Delta_4(G)^*}$ & 
$\bm{\Delta_5(G)^*}$ & $\bm{\Delta_6(G)^*}$ & 
\textbf{Thm~\ref{thm:glazyrin4}~\cite{glazyrin18}} & &
$\bm{n}$ & $\bm{\Delta_3(G)^*}$~\cite{barg14, king16} & $\bm{\Delta_4(G)^*}$ & 
$\bm{\Delta_5(G)^*}$ & $\bm{\Delta_6(G)^*}$ & 
\textbf{Thm~\ref{thm:glazyrin4}~\cite{glazyrin18}} \\\cmidrule{1-6}\cmidrule{8-13}
$125$ & $\bm{1128}$ &$1128$ &$1128$ &$1128$ &$4151$ && $265$ & $49145$ & $19501$ & $7254$ & $\bm{3465}$ & $8797$ \\\cmidrule{1-6}\cmidrule{8-13}
$130$ & $\bm{1128}$ &$1128$ &$1128$ &$1128$ &$4317$ && $270$ & $72667$ & $22466$ & $8584$ & $\bm{3717}$ & $8963$ \\\cmidrule{1-6}\cmidrule{8-13}
$135$ & $1218$ &$\bm{1128}$ &$1128$ &$1128$ &$4482$ && $275$ & $135319$ & $26319$ & $10427$ & $\bm{3998}$ & $9129$ \\\cmidrule{1-6}\cmidrule{8-13}
$140$ & $1387$ &$\bm{1128}$ &$1128$ &$1128$ &$4648$ && $280$ &  & $31427$ & $13008$ & $\bm{4311}$ & $9295$ \\\cmidrule{1-6}\cmidrule{8-13}
$145$ & $1593$ &$\bm{1128}$ &$1128$ &$1128$ &$4814$ && $285$ &  & $36793$ & $13442$ & $\bm{4663}$ & $9461$ \\\cmidrule{1-6}\cmidrule{8-13}
$150$ & $1850$ &$1163$ &$\bm{1128}$ &$1128$ &$4980$ && $290$ &  & $44064$ & $13893$ & $\bm{5062}$ & $9627$ \\\cmidrule{1-6}\cmidrule{8-13}
$155$ & $2178$ &$1262$ &$\bm{1128}$ &$1128$ &$5146$ && $295$ &  & $54538$ & $14363$ & $\bm{5519}$ & $9793$ \\\cmidrule{1-6}\cmidrule{8-13}
$160$ & $2611$ &$1381$ &$1135$ &$\bm{1128}$ &$5312$ && $300$ &  & $70925$ & $14853$ & $\bm{6045}$ & $9959$ \\\cmidrule{1-6}\cmidrule{8-13}
$165$ & $3211$ &$1517$ &$1188$ &$\bm{1128}$ &$5478$ && $305$ &  & $100201$ & $15364$ & $\bm{6660}$ & $10125$ \\\cmidrule{1-6}\cmidrule{8-13}
$170$ & $4098$ &$1670$ &$1271$ &$\bm{1131}$ &$5644$ && $310$ &  &  & $15897$ & $\bm{7386}$ & $10291$ \\\cmidrule{1-6}\cmidrule{8-13}
$175$ & $5199$ &$1846$ &$1361$ &$\bm{1195}$ &$5810$ && $315$ &  &  & $16453$ & $\bm{8257}$ & $10457$ \\\cmidrule{1-6}\cmidrule{8-13}
$180$ & $5582$ &$2051$ &$1458$ &$\bm{1264}$ &$5976$ && $320$ &  &  & $17035$ & $\bm{9322}$ & $10623$ \\\cmidrule{1-6}\cmidrule{8-13}
$185$ & $6006$ &$2290$ &$1564$ &$\bm{1336}$ &$6142$ && $325$ &  &  & $17644$ & $\bm{10653}$ & $10789$ \\\cmidrule{1-6}\cmidrule{8-13}
$190$ & $6477$ &$2575$ &$1679$ &$\bm{1412}$ &$6308$ && $330$ &  &  & $18309$ & $12364$ & $\bm{10955}$ \\\cmidrule{1-6}\cmidrule{8-13}
$195$ & $7005$ &$2919$ &$1805$ &$\bm{1492}$ &$6474$ && $335$ &  &  & $19106$ &  & $\bm{11121}$ \\\cmidrule{1-6}\cmidrule{8-13}
$200$ & $7597$ &$3342$ &$1944$ &$\bm{1578}$ &$6640$ && $340$ &  &  & $20053$ & $17840$ & $\bm{11287}$ \\\cmidrule{1-6}\cmidrule{8-13}
$205$ & $8269$ &$3878$ &$2097$ &$\bm{1668}$ &$6806$ && $345$ &  &  & $21178$ &  & $\bm{11453}$ \\\cmidrule{1-6}\cmidrule{8-13}
$210$ & $9035$ &$4575$ &$2267$ &$\bm{1765}$ &$6972$ && $350$ &  &  & $22494$ & $20168$ & $\bm{11619}$ \\\cmidrule{1-6}\cmidrule{8-13}
$215$ & $9918$ &$5522$ &$2457$ &$\bm{1868}$ &$7138$ && $355$ &  &  & $23893$ &  & $\bm{11785}$ \\\cmidrule{1-6}\cmidrule{8-13}
$220$ & $10946$ &$6880$ &$2670$ &$\bm{1978}$ &$7304$ && $360$ &  &  & $25410$ & $21307$ & $\bm{11951}$ \\\cmidrule{1-6}\cmidrule{8-13}
$225$ & $12158$ &$8548$ &$2911$ &$\bm{2096}$ &$7470$ && $365$ &  &  & $27077$ &  & $\bm{12117}$ \\\cmidrule{1-6}\cmidrule{8-13}
$230$ & $13608$ &$9314$ &$3187$ &$\bm{2223}$ &$7636$ && $370$ &  &  & $28923$ & $22525$ & $\bm{12283}$ \\\cmidrule{1-6}\cmidrule{8-13}
$235$ & $15374$ &$10181$ &$3504$ &$\bm{2359}$ &$7802$ && $375$ &  &  & $30981$ &  & $\bm{12449}$ \\\cmidrule{1-6}\cmidrule{8-13}
$240$ & $17571$ &$11171$ &$3872$ &$\bm{2507}$ &$7968$ && $380$ &  &  & $33291$ & $23833$ & $\bm{12615}$ \\\cmidrule{1-6}\cmidrule{8-13}
$245$ & $20378$ &$12315$ &$4307$ &$\bm{2667}$ &$8134$ && $385$ &  &  & $35904$ &  & $\bm{12781}$ \\\cmidrule{1-6}\cmidrule{8-13}
$250$ & $24090$ &$13652$ &$4827$ &$\bm{2840}$ &$8300$ && $390$ &  &  & $38885$ & $25239$ & $\bm{12947}$ \\\cmidrule{1-6}\cmidrule{8-13}
$255$ & $29230$ &$15238$ &$5460$ &$\bm{3030}$ &$8466$ && $395$ &  &  & $42316$ &  & $\bm{13112}$ \\\cmidrule{1-6}\cmidrule{8-13}
$260$ & $36818$ &$17150$ &$6247$ &$\bm{3237}$ &$8632$ && $400$ &  &  & $46310$ & $26756$ & $\bm{13278}$\\\bottomrule
\end{tabular}
\end{adjustbox}\bigskip
\caption{Upper bounds for $M_{1/7}(n)$ by diverse methods, including new results 
with $\Delta_6(G)^*$. The best bound in each dimension is in boldface.} 
\label{table:values-a7}
\end{table}

\begin{table}[p]\small
\begin{adjustbox}{width=1.25\textwidth,center}
\begin{tabular}{@{}ccccccccccccc@{}}
\toprule
$\bm{n}$ & $\bm{\Delta_3(G)^*}$~\cite{barg14, king16} & $\bm{\Delta_4(G)^*}$ & 
$\bm{\Delta_5(G)^*}$ & $\bm{\Delta_6(G)^*}$ & 
\textbf{Thm~\ref{thm:glazyrin4}~\cite{glazyrin18}} & &
$\bm{n}$ & $\bm{\Delta_3(G)^*}$~\cite{barg14, king16} & $\bm{\Delta_4(G)^*}$ & 
$\bm{\Delta_5(G)^*}$ & $\bm{\Delta_6(G)^*}$ & 
\textbf{Thm~\ref{thm:glazyrin4}~\cite{glazyrin18}} \\\cmidrule{1-6}\cmidrule{8-13}
$225$ & $\bm{3160}$ &$3160$ &$3160$ &$3160$ &$12269$ && $315$ & $18193$ & $8654$ & $4558$ & $\bm{3515}$ & $17176$ \\\cmidrule{1-6}\cmidrule{8-13}
$230$ & $3306$ &$\bm{3160}$ &$3160$ &$3160$ &$12542$ && $320$ & $19112$ & $9815$ & $4836$ & $\bm{3666}$ & $17449$ \\\cmidrule{1-6}\cmidrule{8-13}
$235$ & $3642$ &$\bm{3160}$ &$3160$ &$3160$ &$12814$ && $325$ & $20103$ & $11281$ & $5140$ & $\bm{3825}$ & $17721$ \\\cmidrule{1-6}\cmidrule{8-13}
$240$ & $4035$ &$\bm{3160}$ &$3160$ &$3160$ &$13087$ && $330$ & $21172$ & $13192$ & $5473$ & $\bm{3993}$ & $17994$ \\\cmidrule{1-6}\cmidrule{8-13}
$245$ & $4502$ &$\bm{3160}$ &$3160$ &$3160$ &$13360$ && $335$ & $22330$ & $15784$ & $5840$ & $\bm{4171}$ & $18267$ \\\cmidrule{1-6}\cmidrule{8-13}
$250$ & $5063$ &$\bm{3160}$ &$3160$ &$3160$ &$13632$ && $340$ & $23589$ & $19333$ & $6247$ & $\bm{4360}$ & $18539$ \\\cmidrule{1-6}\cmidrule{8-13}
$255$ & $5752$ &$3196$ &$\bm{3160}$ &$3160$ &$13905$ && $345$ & $24961$ & $20385$ & $6700$ & $\bm{4559}$ & $18812$ \\\cmidrule{1-6}\cmidrule{8-13}
$260$ & $6617$ &$3329$ &$\bm{3160}$ &$3160$ &$14177$ && $350$ & $26464$ & $21506$ & $7207$ & $\bm{4772}$ & $19084$ \\\cmidrule{1-6}\cmidrule{8-13}
$265$ & $7737$ &$3561$ &$\bm{3160}$ &$3160$ &$14450$ && $355$ & $28116$ & $22707$ & $7780$ & $\bm{4998}$ & $19357$ \\\cmidrule{1-6}\cmidrule{8-13}
$270$ & $9243$ &$3824$ &$\bm{3160}$ &$3160$ &$14723$ && $360$ & $29940$ & $23999$ & $8430$ & $\bm{5239}$ & $19630$ \\\cmidrule{1-6}\cmidrule{8-13}
$275$ & $11377$ &$4117$ &$3167$ &$\bm{3160}$ &$14995$ && $365$ & $31965$ & $25395$ & $9177$ & $\bm{5497}$ & $19902$ \\\cmidrule{1-6}\cmidrule{8-13}
$280$ & $13235$ &$4445$ &$3219$ &$\bm{3160}$ &$15268$ && $370$ & $34226$ & $26911$ & $10042$ & $\bm{5774}$ & $20175$ \\\cmidrule{1-6}\cmidrule{8-13}
$285$ & $13816$ &$4815$ &$3317$ &$\bm{3160}$ &$15540$ && $375$ & $36767$ & $28563$ & $11057$ & $\bm{6071}$ & $20448$ \\\cmidrule{1-6}\cmidrule{8-13}
$290$ & $14434$ &$5235$ &$3461$ &$\bm{3160}$ &$15813$ && $380$ & $39642$ & $30373$ & $12263$ & $\bm{6391}$ & $20720$ \\\cmidrule{1-6}\cmidrule{8-13}
$295$ & $15091$ &$5718$ &$3647$ &$\bm{3160}$ &$16086$ && $385$ & $42924$ & $32365$ & $13720$ & $\bm{6737}$ & $20993$ \\\cmidrule{1-6}\cmidrule{8-13}
$300$ & $15791$ &$6277$ &$3849$ &$\bm{3160}$ &$16358$ && $390$ & $46703$ & $34571$ & $15515$ & $\bm{7112}$ & $21265$ \\\cmidrule{1-6}\cmidrule{8-13}
$305$ & $16538$ &$6933$ &$4067$ &$\bm{3235}$ &$16631$ && $395$ & $51103$ & $37026$ & $17784$ & $\bm{7521}$ & $21538$ \\\cmidrule{1-6}\cmidrule{8-13}
$310$ & $17336$ &$7712$ &$4302$ &$\bm{3372}$ &$16904$ && $400$ & $56289$ & $39779$ & $20740$ & $\bm{7966}$ & $21811$\\\bottomrule
\end{tabular}
\end{adjustbox}\bigskip
\caption{Upper bounds for $M_{1/9}(n)$ by diverse methods, including new results 
with $\Delta_6(G)^*$. The best bound in each dimension is in boldface.} 
\label{table:values-a9}
\end{table}

\begin{table}[p]\small
\begin{adjustbox}{width=1.25\textwidth,center}
\begin{tabular}{@{}ccccccccccccc@{}}
\toprule
$\bm{n}$ & $\bm{\Delta_3(G)^*}$~\cite{barg14, king16} & $\bm{\Delta_4(G)^*}$ & 
$\bm{\Delta_5(G)^*}$ & $\bm{\Delta_6(G)^*}$ & 
\textbf{Thm~\ref{thm:glazyrin4}~\cite{glazyrin18}} & &
$\bm{n}$ & $\bm{\Delta_3(G)^*}$~\cite{barg14, king16} & $\bm{\Delta_4(G)^*}$ & 
$\bm{\Delta_5(G)^*}$ & $\bm{\Delta_6(G)^*}$ & 
\textbf{Thm~\ref{thm:glazyrin4}~\cite{glazyrin18}} \\\cmidrule{1-6}\cmidrule{8-13}
$345$ & $\bm{7140}$ &$7140$ &$7140$ &$7140$ &$28011$ && $425$ & $30885$ & $11309$ & $7319$ & $\bm{7140}$ & $34506$ \\\cmidrule{1-6}\cmidrule{8-13}
$350$ & $7426$ &$\bm{7140}$ &$7140$ &$7140$ &$28417$ && $430$ & $31817$ & $12186$ & $7494$ & $\bm{7140}$ & $34912$ \\\cmidrule{1-6}\cmidrule{8-13}
$355$ & $8028$ &$\bm{7140}$ &$7140$ &$7140$ &$28823$ && $435$ & $32789$ & $13185$ & $7730$ & $\bm{7140}$ & $35318$ \\\cmidrule{1-6}\cmidrule{8-13}
$360$ & $8715$ &$\bm{7140}$ &$7140$ &$7140$ &$29229$ && $440$ & $33804$ & $14332$ & $8036$ & $\bm{7140}$ & $35724$ \\\cmidrule{1-6}\cmidrule{8-13}
$365$ & $9506$ &$\bm{7140}$ &$7140$ &$7140$ &$29635$ && $445$ & $34863$ & $15665$ & $8407$ & $\bm{7140}$ & $36130$ \\\cmidrule{1-6}\cmidrule{8-13}
$370$ & $10426$ &$\bm{7140}$ &$7140$ &$7140$ &$30041$ && $450$ & $35971$ & $17232$ & $8808$ & $\bm{7144}$ & $36536$ \\\cmidrule{1-6}\cmidrule{8-13}
$375$ & $11511$ &$\bm{7140}$ &$7140$ &$7140$ &$30447$ && $455$ & $37129$ & $19100$ & $9239$ & $\bm{7190}$ & $36942$ \\\cmidrule{1-6}\cmidrule{8-13}
$380$ & $12809$ &$\bm{7140}$ &$7140$ &$7140$ &$30853$ && $460$ & $38342$ & $21365$ & $9703$ & $\bm{7285}$ & $37348$ \\\cmidrule{1-6}\cmidrule{8-13}
$385$ & $14389$ &$7180$ &$\bm{7140}$ &$7140$ &$31259$ && $465$ & $39613$ & $24170$ & $10205$ & $\bm{7427}$ & $37754$ \\\cmidrule{1-6}\cmidrule{8-13}
$390$ & $16354$ &$7353$ &$\bm{7140}$ &$7140$ &$31665$ && $470$ & $40948$ & $27732$ & $10749$ & $\bm{7618}$ & $38160$ \\\cmidrule{1-6}\cmidrule{8-13}
$395$ & $18866$ &$7692$ &$\bm{7140}$ &$7140$ &$32071$ && $475$ & $42349$ & $32408$ & $11341$ & $\bm{7859}$ & $38566$ \\\cmidrule{1-6}\cmidrule{8-13}
$400$ & $22187$ &$8154$ &$\bm{7140}$ &$7140$ &$32477$ && $480$ & $43823$ & $38495$ & $11986$ & $\bm{8142}$ & $38972$ \\\cmidrule{1-6}\cmidrule{8-13}
$405$ & $26786$ &$8661$ &$\bm{7140}$ &$7140$ &$32883$ && $485$ & $45376$ & $39896$ & $12695$ & $\bm{8442}$ & $39378$ \\\cmidrule{1-6}\cmidrule{8-13}
$410$ & $28304$ &$9221$ &$\bm{7140}$ &$7140$ &$33289$ && $490$ & $47012$ & $41346$ & $13474$ & $\bm{8758}$ & $39784$ \\\cmidrule{1-6}\cmidrule{8-13}
$415$ & $29130$ &$9841$ &$7146$ &$\bm{7140}$ &$33695$ && $495$ & $48741$ & $42853$ & $14337$ & $\bm{9091}$ & $40190$ \\\cmidrule{1-6}\cmidrule{8-13}
$420$ & $29990$ &$10533$ &$7204$ &$\bm{7140}$ &$34100$ && $500$ & $50569$ & $44424$ & $15296$ & $\bm{9443}$ & $40595$\\\bottomrule
\end{tabular}
\end{adjustbox}\bigskip
\caption{Upper bounds for $M_{1/11}(n)$ by diverse methods, including new results 
with $\Delta_6(G)^*$. The best bound in each dimension is in boldface.} 
\label{table:values-a11}
\end{table}

\FloatBarrier

\appendix

\section{Proof of the sufficiency part of Theorem~\ref{teo:main-charac}}
\label{ap:gegenbauer}

We now prove, for the sake of completeness, a theorem that, together
with the linear transformation $L(A_R)$ used to compute the
coordinates of the projection of a vector with respect to an
orthonormal basis of $\spann(R)$, amounts to the sufficiency part of
Theorem~\ref{teo:main-charac}, which is the direction used in this
paper. It is a restatement of a proposition of Musin~\cite[Corollary
3.1]{musin14}.

Recall that $P_l^n$ is the Gegenbauer polynomial of degree $l$ in
dimension $n$ normalized so that $P_l^n(1) = 1$.  For
$0 \leq m \leq n - 2$, $t \in \R$, and $u$, $v \in \R^m$, the
multivariate Gegenbauer polynomial~$P_l^{n,m}$ is the
$(2m+1)$-variable polynomial
\begin{equation}
\label{eq:gegenbauer-multivar-def}
P_l^{n, m}(t,u,v) = \big((1-\|u\|^2)(1-\|v\|^2)\big)^{l/2}P_l^{n-m}\bigg(\frac{t-u 
  \cdot v}{\sqrt{(1-\|u\|^2)(1-\|v\|^2)}}\bigg),
\end{equation}
where $\|v\| = \sqrt{v \cdot v}$. If we use the convention
$\R^0 = \{0\}$, then $P_l^n(t) = P_l^{n,0}(t,0,0)$. Fix $d \geq 0$,
let $\Bcal_l$ be a basis of the space of $m$-variable polynomials of
degree at most~$l$ (e.g.\ the monomial basis), and write $z_l(u)$ for
the column vector containing the polynomials of $\Bcal_l$ evaluated at
$u \in \R^m$. The matrix $Y_l^{n, m}$ is the matrix of polynomials
\[Y_l^{n,m}(t,u,v) =  P_l^{n,m}(t,u,v) z_{d-l}(u)z_{d-l}(v)^{\sf T}.\]

\begin{theorem}[Musin~\cite{musin14}]\label{prop:Hl-kernel}
  Let $R \subseteq S^{n-1}$ with $m = \dim(\spann(R)) \leq n-2$ and
  let~$E$ be an $m \times n$ matrix whose rows form an orthonormal
  basis for $\spann(R)$.  Fix $d \geq 0$ and, for each
  $0 \leq l \leq d$, let $F_l$ be a positive semidefinite matrix of
  size $\binom{d-l+m}{m} \times \binom{d-l+m}{m}$.  Then
  $K \colon S^{n-1} \times S^{n-1} \to \R$ given by
  \[
  K(x,y) = \sum_{l = 0}^d \big\langle F_l, Y_l^{n, m}(x \cdot y, Ex, Ey)\big\rangle
  \]
  is a positive, continuous, and $\Stab{\spann(R)}{\On}$-invariant
  kernel.
\end{theorem}

First we prove that the multivariate Gegenbauer polynomials satisfy
the following positivity property \cite[Theorem 3.1]{musin14}.

\begin{proposition}[Musin~\cite{musin14}]\label{prop:Pk-positive}
  For $0 \leq m \leq n - 2$, let $E$ be an $m \times n$ matrix whose
  rows form an orthonormal set in $\R^n$ and $C$ be a finite subset of
  $S^{n-1}$.  Then, for every nonnegative integer $l$, the matrix
  $ \big( P_l^{n,m}(x \cdot y, Ex, Ey) \big)_{x, y \in C}$ is positive
  semidefinite.
\end{proposition}

\begin{proof}
If $l = 0$ then all polynomials evaluate to $1$ and the proposition
holds, so we assume $l \neq 0$. Let $L$ be the space spanned by the
rows of $E$ and $z$ be a unit vector in $L^\perp$. For each $x \in C$,
write $x = x_1 + x_2$ with $x_1 \in L$ and $x_2 \in L^\perp$. If
$\|x_2\| > 0$, then let $\bar{x} = x_2 / \|x_2\|$, otherwise write
$\bar{x} = z$.
If $\|x_2\|$, $\|y_2\| \neq 0$, then
\[
  \bar{x} \cdot \bar{y} = \frac{x_2\cdot y_2}{\|x_2\|\|y_2\|} =
  \frac{x\cdot y - x_1 \cdot y_1}{\sqrt{(1-\|x_1\|^2)(1-\|y_1\|^2)}}.
\]
Since the rows of $E$ are orthonormal, we have
$x_1 \cdot y_1 = (Ex)\cdot (Ey)$ and hence
$\|x_2\|^l\|y_2\|^lP_l^{n-m}(\bar{x} \cdot \bar{y}) = P_l^{n,m}(x
\cdot y, Ex, Ey)$.

If, say, $\|x_2\| = 0$, then
$\|x_2\|^l\|y_2\|^l P_l^{n-m}(\bar{x} \cdot \bar{y}) = 0$, while
$P_l^{n,m}(x \cdot y, Ex, Ey)$ is also~$0$ as can be seen
from~\eqref{eq:gegenbauer-multivar-def} since
$x \cdot y - x_1 \cdot y_1 = x_2 \cdot y_2 = 0$.

Now $\{\,\bar{x} : x \in C\,\}$ is contained in an embedding of
$S^{n-m-1}$ into~$S^{n-1}$ and by Schoenberg's
theorem~\cite{schoenberg42} we have that
$\big(P_l^{n-m}(\bar{x} \cdot \bar{y})\big)_{x, y \in C}$ is positive
semidefinite. Since $\big(\|x_2\|^l\|y_2\|^l\big)_{x, y \in C}$ is
positive semidefinite, so is
$ \big( \|x_2\|^l \|y_2\|^l P_l^{n-m}(\bar{x} \cdot \bar{y})\big)_{x,
  y \in C}$, and we are done.
\end{proof}

\begin{proof}[Proof of Theorem~\ref{prop:Hl-kernel}]
Since all entries of $Y_l^{n,m}$ are polynomials, $K$ is continuous,
and since $x \cdot y$, $Ex$, and $Ey$ are invariant under the action
of $\Stab{\spann(R)}{\On}$ on $(x,y)$, $K$ is invariant.
To prove positivity, let $C$ be a finite subset of $S^{n-1}$ and
$w \colon C \to \R$ be a function. We have
\[
  \sum_{x, y \in C}w_x w_y K(x,y) = \sum_{l = 0}^d \bigg\langle F_l, \sum_{x,y 
    \in C} w_xw_y Y_l^{n,m}(x \cdot y, Ex, Ey)\bigg\rangle.
\]
To show this quantity is nonnegative, we will show that for all
$l = 0$, \dots,~$d$ the matrix
$\sum_{x, y \in C}w_x w_y Y_l^{n,m}\big(x \cdot y, Ex, Ey \big)$ is
positive semidefinite. For this, write it as a product of matrices:
if~$B$ is the matrix whose columns are given by~$z_{d-l}(Ex)$
for~$x \in C$, then
\begin{align*}
\sum_{x, y \in C}w_x w_y Y_l^{n,m}(x \cdot y, Ex, Ey)
&= \sum_{x, y \in C}w_x w_y z_{d-l}(Ex)z_{d-l}(Ey)^{\sf 
T}P_l^{n,m}(x \cdot y, Ex, Ey)\\
&= B \big(P_l^{n,m}(x \cdot y, Ex, Ey)\big)_{x, y \in C} B^{\sf T},
\end{align*}
and, since the matrix
$\big(P_l^{n,m}(x \cdot y, Ex, Ey)\big)_{x, y \in C}$ is positive
semidefinite by Proposition~\ref{prop:Pk-positive}, we are done.
\end{proof}


\end{document}